\documentclass[SecEq,CM,GP]{degruyter} 

\firstpage{}\volume{}\volumeyear{}\copyrightyear{}\doiyear{}\doi{}\received{XXX}\revised{XXX}

\title[Rolling spheres and the Willmore energy]{Rolling spheres and the Willmore energy}

\author{Felix}{Kn{\"o}ppel}{}{Berlin}
\author{Ulrich}{Pinkall}{}{Berlin}
\author{Peter}{Schr{\"o}der}{}{Pasadena}
\author{Yousuf}{Soliman}{}{Pasadena}

\contact{Mathematics, Technische Universit{\"a}t Berlin, Stra{\ss}e des 17. Juni 135, 10623 Berlin, Germany}{knoeppel@math.tu-berlin.de}
\contact{Mathematics, Technische Universit{\"a}t Berlin, Stra{\ss}e des 17. Juni 135, 10623 Berlin, Germany}{pinkall@math.tu-berlin.de}
\contact{Computing + Mathematical Sciences, California Institute of Technology, 1200 E California Blvd, Pasadena, California 91125, United States of America}{ps@caltech.edu}
\contact{Computing + Mathematical Sciences, California Institute of Technology, 1200 E California Blvd, Pasadena, California 91125, United States of America}{ysoliman@caltech.edu}

\theoremstyle{plain}
 \newtheorem{theorem}{Theorem}[section]

 \newtheorem{definition}[theorem]{Definition}% 
 \newtheorem{remark}[theorem]{Remark}
 \newtheorem{proposition}[theorem]{Proposition}% 
 \newtheorem{lemma}[theorem]{Lemma}% 
 \newtheorem{example}[theorem]{Example}% 
 \usepackage[normalem]{ulem}

 \usepackage{mathtools}%
 \usepackage{tikz}

\newcommand{\eg}{\emph{e.g.}} % "for example"
\newcommand{\ie}{\emph{i.e.}} % "that is"
\renewcommand{\eqref}[1]{Equation~\ref{eq:#1}}
\newcommand{\figref}[1]{Figure~\ref{fig:#1}}
\newcommand{\secref}[1]{Section~\ref{sec:#1}}
\newcommand{\appref}[1]{Appendix~\ref{app:#1}}

\newcommand{\thmref}[1]{Theorem~\ref{thm:#1}}
\newcommand{\lemref}[1]{Lemma~\ref{lem:#1}}

\newcommand{\prpref}[1]{Proposition~\ref{prp:#1}}

\usepackage{etoolbox}
\newcommand{\figloc}[1]{\emph{#1}}

\def\RR{\mathbb{R}} % reals
\def\HH{\mathbb{H}} % quaternions

\newcommand{\Esf}{\mathsf{E}}
\newcommand{\Fsf}{\mathsf{F}}

\newcommand{\Msf}{\mathsf{M}}

\newcommand{\Vsf}{\mathsf{V}}

\renewcommand{\ij}{i\mkern-1.0mu j}
  \newcommand{\jk}{jk}
  \newcommand{\ki}{ki}

  \newcommand{\ijk}{ijk}
  \newcommand{\jil}{jil}

\newcommand{\mesh}{\Msf}
\newcommand{\vertices}{\Vsf}
\newcommand{\edges}{\Esf}
\newcommand{\faces}{\Fsf}

\newcommand{\simplicialsurface}{\mesh=(\vertices,\edges,\faces)}
\newcommand{\fpos}{\mathtt{f}}
\newcommand{\Sij}{\mathtt{S}}
\newcommand{\Cijk}{\mathtt{C}}

\DeclareMathOperator{\End}{End}

\newcommand{\df}{{d\mkern-1mu f}}

\renewcommand{\sl}{\mathfrak{sl}}
\newcommand{\GL}{\mathrm{GL}}

\newcommand{\Sp}{\mathrm{Sp}(1,1)}
\let\sp=\relax
\newcommand{\sp}{\mathfrak{sp}(1,\mkern-1mu 1)}

\newcommand{\HP}{\HH\mathsf{P}}
\DeclareMathOperator{\tr}{tr}
\let\Re=\relax
\DeclareMathOperator{\Re}{Re}
\let\Im=\relax
\DeclareMathOperator{\Im}{Im}

\usepackage{eucal}
\newcommand{\LightCone}{\mathcal{L}}
\newcommand{\Mob}{{\operatorname{M\ddot{o}b}}}
\newcommand{\Spheres}{\mathcal{S}}
\newcommand{\Circles}{\mathcal{C}}
\newcommand{\PointPairs}{\mathcal{P}}
\newcommand{\Cotangent}{\mathcal{V}}

\newcommand{\Willmore}{\mathcal{W}}

\makeatletter
\DeclareFontFamily{OMX}{MnSymbolE}{}
\DeclareSymbolFont{MnLargeSymbols}{OMX}{MnSymbolE}{m}{n}
\SetSymbolFont{MnLargeSymbols}{bold}{OMX}{MnSymbolE}{b}{n}
\DeclareFontShape{OMX}{MnSymbolE}{m}{n}{%
    <-6>  MnSymbolE5
   <6-7>  MnSymbolE6
   <7-8>  MnSymbolE7
   <8-9>  MnSymbolE8
   <9-10> MnSymbolE9
  <10-12> MnSymbolE10
  <12->   MnSymbolE12
}{}
\DeclareFontShape{OMX}{MnSymbolE}{b}{n}{%
    <-6>  MnSymbolE-Bold5
   <6-7>  MnSymbolE-Bold6
   <7-8>  MnSymbolE-Bold7
   <8-9>  MnSymbolE-Bold8
   <9-10> MnSymbolE-Bold9
  <10-12> MnSymbolE-Bold10
  <12->   MnSymbolE-Bold12
}{}

\let\llangle\@undefined
\let\rrangle\@undefined
\let\hmgcoord\@undefined
\DeclareMathDelimiter{\llangle}{\mathopen}%
                     {MnLargeSymbols}{'164}{MnLargeSymbols}{'164}
\DeclareMathDelimiter{\rrangle}{\mathclose}%
                     {MnLargeSymbols}{'171}{MnLargeSymbols}{'171}
\DeclarePairedDelimiterX{\sang}[1]{\langle}{\rangle}{#1}
\DeclarePairedDelimiterX{\dang}[1]{\llangle}{\rrangle}{#1}

\DeclarePairedDelimiterX\braket[2]{\langle}{\rangle}{#1\, \delimsize\vert\, #2}
\makeatother

\newcommand{\hmgpoint}[1]{\mathsf{#1}}
\newcommand{\hmgmatrix}[1]{\mathsf{#1}}
\newcommand{\hmgcoord}[2]{\begin{pmatrix}#1 \\ #2\end{pmatrix}\HH}
\newcommand{\hmgcoords}[2]{\begin{psmallmatrix}#1 \\ #2\end{psmallmatrix}\HH}

\begin{document}

\begin{abstract}
% Insert your abstract here. Remember that online systems rely heavily on the content
% of titles and abstracts to identify articles in electronic bibliographic databases 
% and search engines. We ask you to take great care in preparing the abstract
% and to not use references to the bibliography.
The Willmore energy plays a central role in the conformal
  geometry of surfaces in the conformal 3-sphere \(S^3\). It also arises as the
  leading term in variational problems ranging from black holes,
  to elasticity, and cell biology. In the computational setting a
  discrete version of the Willmore energy is desired. Ideally it
  should have the same symmetries as the smooth formulation. Such a
  M{\"o}bius invariant discrete Willmore energy for simplicial
  surfaces was introduced by Bobenko.

  In the present paper we provide a new geometric interpretation of
  the discrete energy as the curvature of a rolling spheres connection
  in analogy to the smooth setting where the curvature of a connection
  induced by the mean curvature sphere congruence gives the Willmore
  integrand. We also show that the use of a particular projective
  quaternionic representation of all relevant quantities gives clear
  geometric interpretations which are manifestly M{\"o}bius invariant.
\end{abstract}

% \acknowl{Insert acknowledgments of the assistance of colleagues or similar notes of appreciation here.}

%% \tableofcontents %% Just for papers exceeding 50 pages.

% \begin{tabular}{r|l}
%     Column width & \the\columnwidth \\
%     Text width & \the\textwidth \\
%     Font size & \makeatletter\f@size pt \\
%  \end{tabular}

\section{Introduction}\label{sec:intro}

The Willmore energy is the most well-studied example of a surface
functional that is invariant under M{\"o}bius transformations. For an
immersion \(f : M\to\RR^3\) of a smooth two-dimensional manifold, the
Willmore energy is defined as
\begin{equation}
  \mathcal{W} = \int_{M} (H^2 - K)\,\sigma_f.
\end{equation}
Here \(H = \tfrac12(\kappa_1+\kappa_2)\) is the mean curvature,
\(K=\kappa_1\kappa_2\) the Gaussian curvature, \(\sigma_{f}\) the induced
volume form, and \(\kappa_1\) and \(\kappa_2\) are the principal
curvatures of the immersion, resulting in a M{\"o}bius invariant
integrand. Recall that the group of M{\"o}bius transformations of
\(S^3\) consists of the diffeomorphisms of \(S^3\) carrying round 2-spheres into round
2-spheres. By adding a point at \(\infty\) we can consider \(\RR^3\)
as a conformal submanifold of the one-point compactification
\(S^3=\RR^3\cup\{\infty\}\) and treat M{\"o}bius transformations as
conformal diffeomorphisms of \(\RR^3\cup\{\infty\}\). The group of
M{\"o}bius transformations of \(\RR^3\) fixing \(\infty\) is generated by the orthogonal transformations (isometries fixing the
origin), homeotheties (\(p\in\RR^3\mapsto \lambda p\),
\(\lambda > 0\)), and translations. Also appending inversion in the
unit sphere to this list, the set of transformations generates
the full group of M{\"o}bius transformations of \(S^3\).

\begin{figure}[t]
  \centering
  \includegraphics[width=\columnwidth]{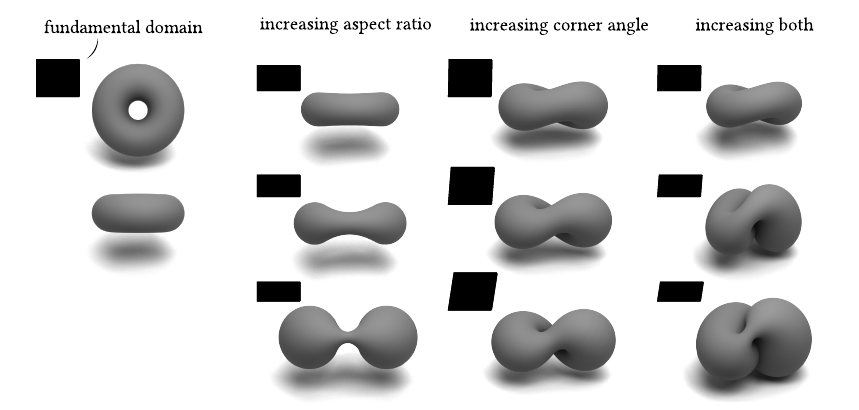}
  \caption{The discrete formulation of the Willmore energy enables the computation of (constrained) critical points via numerical minimization. The discrete constrained Willmore tori visualized here provide realizations of points in Teichm{\"u}ller space as discrete surfaces in \(S^3\).}
  \label{fig:TorusTeichmuller}
\end{figure}

% Motivation:
In the smooth setting the Willmore energy, due to its invariance under
M{\"o}bius transformation of \(S^3\), plays a central role in
conformal geometry and has stimulated many interesting research
directions~\cite{Li:1982:NCI,Bohle:2008:CWS,Quintino:2021:CWS}. In
2012, Marques and Neves used the Almgren-Pitts min-max theory
to resolve the celebrated Willmore conjecture stating that, up to
M{\"o}bius transformation, the Clifford torus minimizes the Willmore
energy among immersed tori in \(\RR^3\)~\cite{Marques:2014:MMT}. The Willmore energy also
arises in a variety of scientific domains. In cell biology,
it models
the geometry and locomotion of lipid
bilayers~\cite{Lipowsky:1991:CM,Canham:1970:MEB,Helfrich:1973:EPL,Evans:1974:BRC,Evans:2010:SJV}. In
general relativity, it
  appears as the leading term of the Hawking
mass~\cite{Hawking:1968:GRE,Koerber:2021:APW}. In nonlinear
elasticity, it measures the
bending energy of thin plates~\cite{Friesecke:2002:TGR}. The Willmore
energy is also popular in computational geometry and computer graphics
due to the regularizing effects of its gradient
flow~\cite{Bobenko:2005:DWF,Soliman:2021:CWS,Rumpf:2004:LSF,Riviere:2008:AAW}. Constrained
Willmore surfaces also appear in a
theory that encapsulates minimal surfaces, CMC surfaces, and Willmore
surfaces~\cite{Quintino:2021:CWS}---see \figref{TorusTeichmuller} for numerical examples of constrained Willmore tori computed following the approach in \cite{Soliman:2021:CWS}. As the importance of the Willmore energy is
evidenced in the context of differential geometry and physics, it is
desirable to have a rich theory of discrete Willmore surfaces.

\paragraph{Discrete Setup}
We study the discrete differential geometry of the M{\"o}bius
invariant discretization of the Willmore energy on a simplicial
surface
\(\simplicialsurface\)~\cite{Bobenko:2005:CES,Bobenko:2005:DWF},
where \(\vertices\) denotes the set of vertices \(i\), \(\edges\) the set of
oriented edges \(\ij\), and \(\faces\) the set of
oriented triangles \(\ijk\). An
immersion into \(S^3\) is given by vertex
positions \(\fpos_i\in S^3\) for \(i\in \vertices\).  The discrete energy \(\Willmore\) then is
defined in terms of the intersection angles \( \beta_{\ij} \) between
the circumcircles \(\Cijk_{\ijk}\) and \(\Cijk_{\jil}\) of adjacent faces
\(ijk\) and \(jil\) and using the negative angle defect to measure the
planarity of each vertex star after sending \(\fpos_i\) to infinity with a
M{\"o}bius transformation:
\begin{equation}
    \Willmore \coloneqq \tfrac12\sum_{i\in\vertices}\Willmore_i,\qquad  \Willmore_i := \sum_{\ij}\beta_{\ij} - 2\pi.
    \label{eq:DiscreteWillmore}
\end{equation}
\begin{figure}[!h]
    \centering
    \begin{tikzpicture}
        \node[anchor=south west,inner sep=0] at (0,0) {\includegraphics{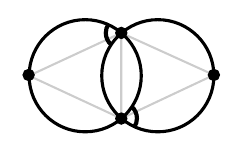}};
        \node[anchor=south west,inner sep=0] at (-.05,1.1) {\(\fpos_l\)};
        \node[anchor=south west,inner sep=0] at (3.8,1.1) {\(\fpos_k\)};
        \node[anchor=south west,inner sep=0] at (1.9,2.25) {\(\fpos_i\)};
        \node[anchor=south west,inner sep=0] at (1.9,-.05) {\(\fpos_j\)};
    \end{tikzpicture}
\end{figure}

Since the discrete energy is defined using only the angles between
circles it is M{\"o}bius invariant.  Furthermore, it is non-negative and
\( \Willmore_i + K_i \geq 0 \) where \(K_i\) is the usual
discrete Gaussian curvature (angle defect). Finally it vanishes
when the vertex star is convex and \(\fpos_i\) and all its
neighbors lie on a common sphere.  These properties of \(\Willmore_i\) mirror the
smooth setting in that \( (H^2-K)\,\sigma_f \) is non-negative, M{\"o}bius invariant, and measures
infinitesimal sphericality (\( \kappa_1 = \kappa_2 \)).  Together, these
properties justify calling it the discrete Willmore energy.  Recently,
\(\Gamma\)-convergence of the functional has been established with
respect to weak-\(*\) convergence in \(W^{1,\infty}\) as graphs of the
piecewise linear surface to the smooth
surface~\cite{Gladbach:2023:AWE}. Consistency of the local
approximation of the energy is also known for triangulations aligned
to principal curvature directions~\cite{Bobenko:2008:SfC}.

\paragraph*{Rolling Mean Curvature Spheres: Smooth and Discrete}
To express the Willmore energy in a manifestly M{\"o}bius invariant way, one introduces the \textbf{conformal Gauss map}, also known as the mean curvature sphere congruence:
\begin{equation}
    S : M \to \Spheres{}\subset\RR^{4,1},\qquad \text{satisfying }T_{f(p)}S_p = df(T_{p}M)\text{ and }H_{S_p} = H_{p}, \label{eq:ConformalGaussMap}
\end{equation}
where \(\Spheres{}\) is the Lorentzian space of 2-spheres in \(S^3\) (see \secref{SpaceOfSpheres}), \(p\in M\), \(H_{S_p}\) the mean curvature of the sphere \(S_p\), and \(H_{p}\) the mean curvature of the immersion \(f\) at \(p\)~\cite{Blaschke:1929:VUD,Bryant:1984:DTW,Burstall:2010:CSG}. The conformal Gauss map is M{\"o}bius invariant, even though the mean curvature itself is not. A classical result due to Blaschke is that the Willmore energy is equal to the area of the conformal Gauss map~\cite{Blaschke:1929:VUD}. A modern treatment of conformal submanifold geometry using the machinery of Cartan geometries was given in~\cite{Sharpe:2000:DGC,Burstall:2010:CSG}, and Sharpe showed that the Willmore energy can be realized as the curvature of a Cartan geometry obtained by restricting the flat M{\"o}bius structure of \(S^3\) to an immersed surface~\cite{Sharpe:2000:DGC}.

% \begin{figure}
%   \includegraphics[width=.5\columnwidth]{figures/RollingTangentPlanes.pdf}
%   \includegraphics[width=.5\columnwidth]{figures/TangentPlaneCurvature.pdf}
%   \caption{Rolling Tangent Planes over a smooth surface. Here we visualize the parallel transport induced by rolling tangent planes over a smooth surface. Notice that the trajectories of the parallel transport are orthgonal to the tangent planes. On the right hand side of the figure we visualize the curvature obtained by rolling the tangent planes around a closed curve: the curvature is equal to a rotation about the normal with an angle determined by the Gauss curvature of the surface.}
% \end{figure}

To elucidate the differential geometric interpretation of the discrete
energy, we first develop a related geometric interpretation of
the Willmore energy in the smooth setting.  The main idea is that the
Willmore integrand can be realized as the curvature of a rolling
spheres connection in the same way as the Gauss curvature form can be
realized as the curvature of the Levi-Civita
connection---geometrically, the Levi-Civita connection describes how
to roll tangent planes over the surface without slipping or
twisting. Extrinsically, the process of rolling tangent planes is
characterized by the trajectories of the induced
parallel transport being orthogonal to the tangent planes themselves. By
following the parallel transport around a closed loop one obtains an
affine map of the tangent plane onto itself. The curvature of the
connection at a point \(p\) is the affine map of \(T_pM\) onto itself
obtained as the limit of following the parallel transport around
infinitesimal loops based at \(p\). As the Levi-Civita connection is
torsion free, the curvature fixes the point \(p\) and on the tangent
space \(df(T_pM)\) acts by a rotation around the normal with an angle
given by the Gauss curvature of the surface.

To obtain a M{\"o}bius geometric generalization, we replace the
tangent plane congruence with an arbitrary tangent sphere congruence
and replace the Levi-Civita connection with a M{\"o}bius connection
whose parallel transport trajectories are orthogonal to the sphere
congruence. This orthogonal trajectories property, visualized in
\figref{Tractrix}, justifies calling the connection a rolling
spheres connection since the trajectories follow the paths that one
would get by rolling a sphere over the surface in M{\"o}bius three-space. In the
special case when the sphere congruence is given by the M{\"o}bius
invariant mean curvature sphere congruence, we find that the Willmore
energy arises as the rotational component of the curvature form of the
connection obtained by rolling the mean curvature spheres over the
surface.  Our main result about the discrete energy provides an
analogous geometric interpretation of it: we prove that the discrete
Willmore energy can be computed from the curvature of a M{\"o}bius
connection obtained by rolling the circumspheres from one edge to the
next around a vertex.

\begin{figure}[h]
  \includegraphics[width=\columnwidth]{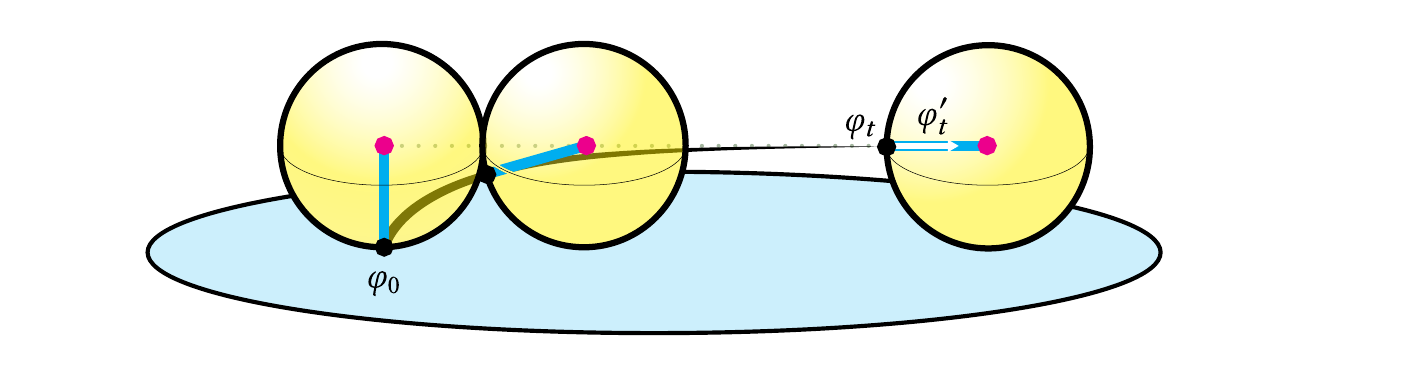}
  \caption{M{\"o}bius geometric rolling of a sphere congruence traces out trajectories that intersect the spheres orthogonally. The tractrix curve visualized above is obtained by rolling a sphere of constant radius over a flat plane.}
  \label{fig:Tractrix}
\end{figure}

\paragraph{Quaternionic Projective Geometry and \(S^3\)}
In \secref{MobiusGeometryWQuaternions} we review the basics of the quaternionic projective (\ie, M{\"o}bius) geometry of \(S^4\) as presented in~\cite{Burstall:2004:CGS} and specialize the situation to \(S^3\) by restricting the group of quaternionic projective transformations of \(S^4\) to those which preserve a fixed M{\"o}bius 3-sphere, \(S^3\), in \(S^4\). 
We then realize spaces of geometric objects associated to this quaternionic projective geometry (\eg, the spaces of oriented 2-spheres, circles, and point pairs in \(S^3\)) as spaces of quaternionic matrices.
We define an algebraic structure, resembling vector calculus in
\(\RR^3\), on the space of \(p\)-spheres in \(S^3\) for
\(p\in\{0,1,2\}\). The vector calculus of \(p\)-spheres involves a cross
product and a dot product that act on the algebraic representations of
oriented \(p\)-spheres in \(S^3\). The
cross-product of two intersecting circles, for example, produces the sine of their intersection angle multiplying the unique circle orthogonal to both circles. On the other hand, their dot product produces the cosine of the intersection angle. In \secref{RollingSphereConnections} we use this quaternionic formalism to describe the geometry of rolling spheres in both the smooth and discrete settings using quaternionic connections.
The quaternionic formalism we present here should be of independent
geometric interest---our experience is that it is
geometrically meaningful as well as quite efficient and easily implemented on a
computer for the algorithmic manipulation of spheres and circles. 

Recently, there have been a several authors who have studied fundamental properties of quaternionic M{\"o}bius transformations. Fixed points and conjugacy classes of M{\"o}bius transformations of \(S^3\) have been characterized in~\cite{Jakobs:2010:MTR,Bisi:2009:MTP}. Quaternionic holomorphic geometry provides an elegant description of surfaces in 3-space and 4-space that has been especially fruitful in the study of Willmore surfaces~\cite{Heller:2014:ECW,Burstall:2004:CGS}. In this context the quaternionic realization of the space of oriented 2-spheres plays a central role.

\section{M{\"o}bius Geometry of \(S^3\) with Quaternions}
\label{sec:MobiusGeometryWQuaternions}
It is natural to study the discrete Willmore energy in the M{\"o}bius
geometry of \(S^3\) for which we will employ a quaternionic projective
model. We begin with a review of the quaternionic projective geometry of \(\HP^1\) and by restricting the group of M{\"o}bius transformations to those fixing a three-sphere \(S^3\) inside of \(S^4\) we obtain a quaternionic projective model of the M{\"o}bius geometry of \(S^3\). 

\subsection{M{\"o}bius transformations preserving \(S^3\)}
The quaternionic projective model of \(S^4\) is based on the
observation that \(\HP^1\) is conformally equivalent to
\(S^4\)~\cite[Section 3.1]{Burstall:2004:CGS}. The group
\(\GL(2;\HH)\) acts on \(\HP^1\) by \(\hmgmatrix{A}(\psi\HH) = (\hmgmatrix{A}\psi)\HH\),
\(\hmgmatrix{A}\in\GL(2;\HH)\), and acts by M{\"o}bius transformations on
\(S^4\). Moreover, all orientation preserving conformal
differeomorphisms are realized this way and the representation of a
M{\"o}bius transformation is unique up to a real
scale~\cite{Kulkarni:1988:CG}.

A three-sphere inside of \(S^4\) is determined by the isotropic lines
of an indefinite Hermitian inner product on \(\HH^2\)~\cite[Chapter
10]{Burstall:2004:CGS}. We fix the indefinite Hermitian inner
product
\begin{equation}
  \sang{\begin{pmatrix}\psi_0\\\psi_1\end{pmatrix},\begin{pmatrix}\varphi_0\\
      \varphi_1\end{pmatrix}}= \bar{\psi}_0\varphi_1 + \bar{\psi}_1\varphi_0.
\end{equation}
The isotropic lines form a conformal model of \(S^3\): 
\begin{align} 
  S^3 & \coloneqq \big\{ \psi\HH\in \HP^1 \,\,\mid\,\, \sang{\psi,\psi} = 0\big\} \\
  & = \Big\{\hmgcoord{p}{1}\,\,\mid\,\, p\in\RR^3\cong\Im\HH\Big\}\cup\Big\{\infty\Big\}.
\end{align}
A unitary transformation \(\hmgmatrix{A}\) is an
automorphism of \((\HH^{2},\sang{\cdot,\,\cdot})\), \ie{} an
invertible linear map \(\hmgmatrix{A}\in\GL(2;\HH)\) satisfying
\(\sang{\hmgmatrix{A}\psi,\hmgmatrix{A}\varphi} =
\sang{\psi,\varphi}\) for all \(\psi,\varphi\in\HH^{2}\), and we denote
the group of all such transformations as \(\Sp\). They satisfy
\(\hmgmatrix{A}^*\hmgmatrix{A} = \hmgmatrix{I}\) where
\(\hmgmatrix{I}\) is the identity matrix and \(\hmgmatrix{A}^*\) the
adjoint with respect to the indefinite Hermitian form
\begin{equation} 
  \hmgmatrix{A}^*  = \begin{pmatrix}
    \bar{d} & \bar{b} \\ \bar{c} & \bar{a}\end{pmatrix}\quad \text{for}\quad \hmgmatrix{A} = \begin{pmatrix}a & b \\ c & d\end{pmatrix}.
\end{equation}
Evidently the action of \(\Sp\)
preserves \(S^3\) and since the restriction of a conformal map of
\(S^4\) to \(S^3\) is still conformal, we deduce that elements of
\(\Sp\) can be identified with orientation preserving
M{\"o}bius transformations of \(S^3\).

These transformations can be written in terms of simpler, geometrically
meaningful components associated with the generators of the group of
M{\"o}bius transformations (see \appref{SpDecomp} for a proof).
\begin{proposition}
  Let \(\hmgmatrix{A}\in\Sp\). Then there exists unique \(x,y\in\RR^3\) and \(\mu\in\HH\) satisfying 
  \begin{equation}
  \hmgmatrix{A} = \begin{pmatrix} 1 & 0 \\ y & 1\end{pmatrix} \begin{pmatrix}\mu & 0 \\ 0 & \bar{\mu}^{-1}\end{pmatrix} \begin{pmatrix} 1 & x \\ 0 & 1 \end{pmatrix}.
  \end{equation}
  \label{prp:SpDecomp}
\end{proposition}
The factors have straightforward geometric
interpretations. For \(x\in\RR^3\) the matrix
\begin{equation}
\hmgmatrix{T}_x \coloneqq \begin{pmatrix}1 & x \\ 0 & 1\end{pmatrix},\qquad \hmgmatrix{T}_x \begin{pmatrix}p \\ 1\end{pmatrix}\HH = \begin{pmatrix}p + x \\ 1\end{pmatrix}\HH
\end{equation}
describes the translation by \(x\). For \(\mu\in\HH\) the matrix 
\begin{equation}
\hmgmatrix{R}_{\mu} \coloneqq \begin{pmatrix}\mu & 0 \\ 0 & \bar{\mu}^{-1}\end{pmatrix},\qquad \hmgmatrix{R}_\mu\begin{pmatrix}p \\ 1\end{pmatrix}\HH = \begin{pmatrix}\mu\,p\,\bar{\mu} \\ 1 \end{pmatrix}\HH
\end{equation}
describes the stretch rotation given by conjugation with \(\mu\). Lastly, for \(y\in\RR^3\) the matrix 
\begin{equation}
\hat{\hmgmatrix{T}}_y \coloneqq \begin{pmatrix} 1 & 0 \\ y & 1 \end{pmatrix},\qquad \hat{\hmgmatrix{T}}_y \begin{pmatrix}p \\ 1\end{pmatrix}\HH = \begin{pmatrix}(p^{-1} + y)^{-1} \\ 1 \end{pmatrix}\HH
\end{equation}
describes the inversion in the unit sphere, composed with a
translation by \(-y\), followed by another inversion in the unit
sphere. Analogous to how a translation shifts the origin, the
transformation \(\hat{\hmgmatrix{T}}_y\) shifts the infinity point.

\begin{remark}
    Each of the terms, \(\hat{\hmgmatrix{T}}_y\), \(\hmgmatrix{R}_\mu\),
    \(\hmgmatrix{T}_x\), in \prpref{SpDecomp}
    describes an orientation preserving M{\"o}bius transformations of
    \(S^3\), hence the group \(\Sp\) contains only
    orientation preserving M{\"o}bius transformations of \(S^3\). Specifically, inversion in a 2-sphere is
    not an orientation preserving M{\"o}bius transformation and thus not
    representable by an element of \(\Sp\).
\end{remark}

To obtain all M{\"o}bius transformations of \(S^3\) one needs to consider the larger group
\begin{equation}
  \Mob(3) \coloneqq \Big\{ \hmgmatrix{A}\in\HH^{2\times 2}~\mid~\hmgmatrix{A}^*\hmgmatrix{A} = \pm\hmgmatrix{I}\Big\}.
\end{equation}
The orientation reversing M{\"o}bius transformations of \(S^3\) are those elements satisfying \(\hmgmatrix{A}^*\hmgmatrix{A} =
-\hmgmatrix{I}\).  

\begin{remark}
    Both \(\Sp\) and \(\Mob(3)\) are double covers of the group of orientation preserving and all M{\"o}bius transformations of \(S^3\), respectively, and so one should consider the quotients of these groups by the order two subgroup \(\{+\hmgmatrix{I},-\hmgmatrix{I}\}\) if working with the double cover directly is not desired.
\end{remark}

The group \(\Mob(3)\) has two connected components with \(\Sp\) being
path connected to the identity. Therefore, the Lie algebra of
\(\Mob(3)\) is equal to the Lie algebra of \(\Sp\) denoted
\begin{equation}
  \sp = \big\{\hmgmatrix{Y}\in\HH^{2\times 2}\,\,\mid\,\,\hmgmatrix{Y}^* + \hmgmatrix{Y} = 0\big\}.
\end{equation}
Elements of the Lie algebra \(\hmgmatrix{Y}\in\sp\) can be integrated into a 1-parameter family of M{\"o}bius transformations
\begin{equation}
  t \mapsto \exp(t \hmgmatrix{Y})
\end{equation} using the exponential map 
\begin{equation}
  \exp:\sp\to\Sp\qquad \exp \hmgmatrix{A}\coloneqq \sum_{n=0}^{\infty}\frac{\hmgmatrix{A}^n}{n!}.
\end{equation}
Elements of the Lie algebra describe infinitesimal M{\"o}bius transformations of \(S^3\). The vector field associated with \(\hmgmatrix{Y}\in\sp\) is given by
\begin{equation}
  \mathsf{V}_{\hmgmatrix{Y}}(\hmgpoint{p}) = \left.\frac{d}{dt}\right|_{t=0}\exp(t\hmgmatrix{Y})\cdot\hmgpoint{p}.
\end{equation}
In this way \(\sp\) is identified with the algebra of M{\"o}bius vector fields of \(S^3\).

Two important operations on matrices are the inner product on
\(\HH^{2\times 2}\):
\begin{equation}
  \langle\hmgmatrix{A},\hmgmatrix{B}\rangle = -\tfrac12\tr_{\RR}(\hmgmatrix{AB}).
\end{equation} 
and the cross
product on \(\HH^{2\times 2}\) defined via the commutator:
\begin{equation}
  \hmgmatrix{A}\times\hmgmatrix{B}\coloneqq \tfrac12[\hmgmatrix{A},\hmgmatrix{B}] = \tfrac12(\hmgmatrix{AB} - \hmgmatrix{BA}),
\end{equation}
for \(\hmgmatrix{A},\hmgmatrix{B}\in\HH^{2\times 2}\).
Conveniently, the product of two matrices
\(\hmgmatrix{A},\hmgmatrix{B}\in\HH^{2\times 2}\) can be written as
one half the sum of
the commutator and anti-commutator:
\begin{equation}
  \hmgmatrix{A}\hmgmatrix{B} = \tfrac12\{\hmgmatrix{A},\hmgmatrix{B}\} + \hmgmatrix{A}\times \hmgmatrix{B}.
\end{equation}
For matrices describing inversions in spheres half the anti-commutator
equals minus the inner product (see Remark \ref{rem:anticommutator}).

\begin{figure}[!h]
    \centering
    \begin{tikzpicture}
        \node[anchor=south west,inner sep=0] at (0,0) {\includegraphics[width=\textwidth]{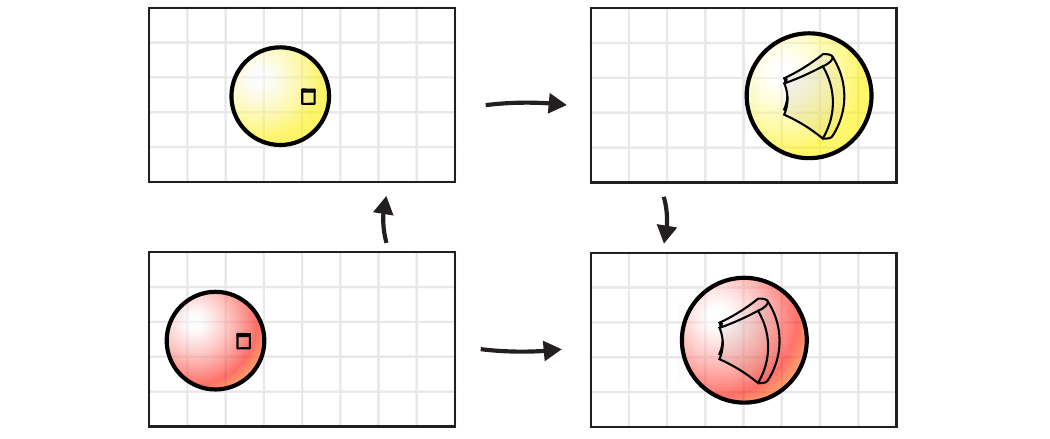}};
        \node[anchor=south west,inner sep=0] at (5.6,2.9) {\small\(\hmgmatrix{A}^{-1}\)};
        \node[anchor=south west,inner sep=0] at (8.5,2.9) {\small\(\hmgmatrix{A}\)};
        \node[anchor=south west,inner sep=0] at (7.05,4) {\small\(\hmgmatrix{B}\)};
        \node[anchor=south west,inner sep=0] at (7.05,1.6) {\small\(\tilde{\hmgmatrix{B}}\)};
    \end{tikzpicture}
    \caption{Conjugating a matrix \(\hmgmatrix{B}\) by a M{\"o}bius transformation \(\hmgmatrix{A}\) corresponds to transforming the inputs and outputs of the M{\"o}bius transformation described by \(\hmgmatrix{B}\) into new coordinates determined by \(\hmgmatrix{A}\).}
    \label{fig:ConjugationAction}
\end{figure}
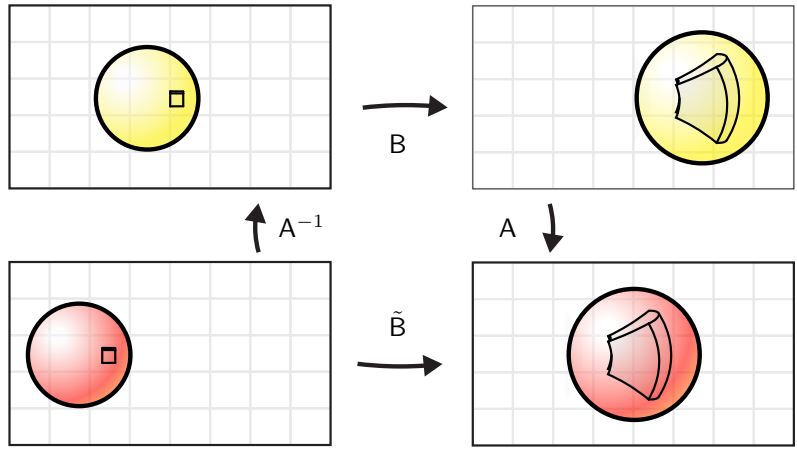

To conclude this section, we note that when transforming the
coordinates by a M{\"o}bius transformation \(\hmgmatrix{A}\in\Mob(3)\) the
action of another M{\"o}bius transformation \(\hmgmatrix{B}\in\Mob(3)\)
must be conjugated by \(\hmgmatrix{A}\) to describe the same
M{\"o}bius transformation in the new coordinates:
\begin{equation}
  \tilde{\hmgmatrix{B}} = \hmgmatrix{A}\,\hmgmatrix{B}\,\hmgmatrix{A}^{-1}.
\end{equation}
See \figref{ConjugationAction} for an illustration of the action of
conjugation by a M{\"o}bius transformation. This
  relationship is of great practical importance. Any argument
  involving M{\"o}bius invariant quantities can be answered assuming a
  convenient M{\"o}bius transformation. For example, the angle between
  two intersecting spheres is the dihedral angle between the two
  planes that result from sending a point common to both spheres to
  infinity.

\subsection{Spaces of Spheres in \(S^3\)}
\label{sec:SpaceOfSpheres}
In this section we describe how to realize the space of oriented
\(p\)-spheres for \(p\in\{0,1,2\}\) inside of \(\HH^{2\times 2}\) in a
geometric way by associating M{\"o}bius transformations with each
\(p\)-sphere in \(S^3\). These constructions are motivated by the
natural association of \(p\)-spheres associated with cells of a
simplicial surface: a circumsphere for each unoriented edge, a
circumcircle for each face, and a point pair for each oriented
edge. In \secref{MobiusSpheresAlgebra} we use the algebraic structure
on \(\HH^{2\times 2}\) to manipulate these \(p\)-spheres algebraically
in a geometrically intuitive way.

To wit, each oriented 2-sphere is
associated with the M{\"o}bius transformation of inversion in that
sphere yielding the ``space of spheres'' inside \(\HH^{2\times 2}\). 
With each oriented circle in \(S^3\) we define the inversion in the circle as
the M{\"o}bius transformation that when restricted to any 2-sphere
containing the circle is an inversion in the circle in \(S^2\). It can
also be computed by composing the inversions in spheres orthogonally
intersecting in the circle. Describing this circle inversion as an
element of \(\HH^{2\times 2}\) realizes the ``space of circles''
inside of \(\HH^{2\times 2}\). Finally, the inversion in an oriented point pair
is the M{\"o}bius transformation that when restricted to any circle
containing the pair of points is the inversion in a pair of points in
\(S^1\). It is equal to the Euclidean inversion \(x\mapsto -x\) after
sending the pair of points to the origin and infinity
respectively. This realizes the space of oriented point pairs inside
\(\HH^{2\times 2}\). In these cases, the sign of the corresponding matrix will determine the orientation of the sphere. 
We visualize and describe in more detail all of
these M{\"o}bius transformations in the following sections.

\paragraph{Space of 2-spheres}
Spheres in \(S^3\) are either
    2-spheres or 2-planes in \(\RR^3\), depending on whether the
    sphere passes through the point at infinity.  To determine the
    conditions for \(\hmgmatrix{S}\in\HH^{2\times 2}\) to
    describe the inversion in an oriented sphere we 
    consider the inversion in an oriented sphere \(\hat{S}\subset\RR^3\) with mean curvature \(h\) and passing through the origin where it has unit
    normal \(n\in S^2\subset\RR^3\). 
    When \(h\neq 0\) the center of
    \(\hat{S}\) is given by
\begin{equation}
  m = - \frac{1}{h}n
\end{equation}
and so the sphere inversion \(\tilde{S}\) about \(\hat{S}\) is given by
\begin{equation}
  \tilde{S} : \RR^3\cup\{\infty\}\to \RR^3\cup\{\infty\},\qquad \tilde{S}(y) = \begin{cases}
    y - 2\sang{n,y} n & \text{if }h=0,\\
    m + \frac{1}{h^2}\frac{y-m}{\|y-m\|^2} & \text{if }h\neq 0.
  \end{cases}
\end{equation}
\begin{figure}[ht]
  \centering
  \includegraphics[width=\textwidth]{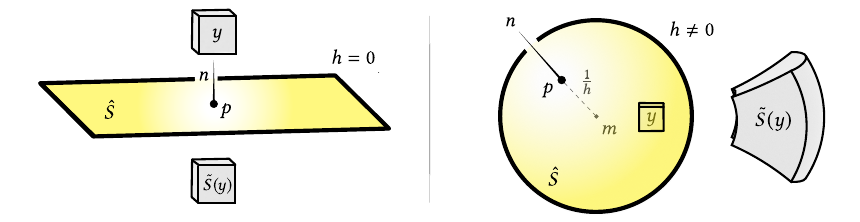}
  \caption{The inversion in a 2-sphere is well described in terms of a point \(p\) on the sphere, its normal at \(p\) and its mean curvature \(h\). This description works uniformly to describe inversions in round spheres or reflections through planes.}
\end{figure}

This sphere inversion can be represented as matrix multiplication by
\begin{equation}
  \hmgmatrix{S} = 
  \begin{pmatrix}
    n & 0 \\ -h & -n
  \end{pmatrix}.
\end{equation}
Indeed, it is easy to verify that 
\begin{equation}
  \hmgmatrix{S}\begin{pmatrix}
    y \\ 1
  \end{pmatrix} = \begin{pmatrix}
    \tilde{S}(y) \\ 1
  \end{pmatrix}(-hy - n).
\end{equation}
By conjugating \(\hmgmatrix{S}\) by the M{\"o}bius transformation \(\hmgmatrix{T}_p\) describing translation by \(p\) we deduce that the inversion in the oriented sphere in \(\RR^3\) with mean curvature \(h\) that passes through the point \(p\in\RR^3\) and at \(p\) has unit normal \(n\) is represented by the matrix
\begin{equation}
  \hmgmatrix{S} = \hmgmatrix{T}_p\begin{pmatrix} n & 0 \\ -h & -n\end{pmatrix} \hmgmatrix{T}_p^{-1}.
\end{equation}
A direct computation shows that all of the matrices describing sphere
inversions satisfy \(\tr_{\RR}\hmgmatrix{S} = 0\),
\(\hmgmatrix{S}^* = \hmgmatrix{S}\), and
\(\hmgmatrix{S}^2 = -\hmgmatrix{I}\). Conversely, these three equations
characterize all of the inversions in 2-spheres in \(S^3\).  In
\appref{QSpheres}, we show that any such matrix describes an oriented
2-sphere in \(S^3\) and that we can identify the space
\begin{equation}
  \Spheres \coloneqq \big\{\hmgmatrix{S}\in\HH^{2\times 2}\,\,\mid\,\, \tr_{\RR}\hmgmatrix{S} = 0,\, \hmgmatrix{S}^* = \hmgmatrix{S},\, \hmgmatrix{S}^2 = -\hmgmatrix{I}\big\} \subset \Mob(3).
\end{equation}
with the space of oriented \(2\)-spheres in \(S^3\). 

\begin{remark} 
  One can recover the sphere from the eigenlines of the linear
  operator \(\hmgmatrix{S}\): that is,
  \(\hat{S} = \{\psi\HH\in S^3\,\,\mid\,\, \sang{\hmgmatrix{S}\psi,\psi} =
  0\}\). Moreover,
  \(B(\hmgmatrix{S}) = \{\psi\HH\in S^3\,\,\mid\,\,
  \sang{\hmgmatrix{S}\psi,\psi} < 0\}\) is an open ball in \(S^3\). Both \(\hmgmatrix{S}\) and \(-\hmgmatrix{S}\) determine the same round 2-sphere, but they bound complimentary round balls in \(S^3\) since \(\overline{B(-\hmgmatrix{S})} = S^3\setminus
  B(\hmgmatrix{S})\). Taking this round ball bounded by \(\hat{S}\) as its interior determines the orientation of the 2-sphere. 
\end{remark}

\paragraph{Space of circles}
Just as we identified spheres with the inversion in them we now
identify circles with the M{\"o}bius transformation which, restricted
to any 2-sphere containing the circle is an inversion in the circle. Since this is a
M{\"o}bius geometric definition, we may transform the circle into a
line by sending a point on the circle to infinity. Any sphere
containing the circle will then be sent to a plane containing the
line.  The definition of the inversion in the line states that the
restriction to any plane containing the line is equal to the
reflection about the line. In other words, the inversion in a line is
\(180^\circ\) rotation around the line. By reversing the M{\"o}bius
transformation mapping the circle into the line we see that the
inversion in a circle can also be thought of as a \(180^\circ\)
rotation around the circle. \figref{CircleRotation} (\figloc{left})
illustrates the \(180^\circ\) degree rotation around an oriented line
in \(\RR^3\) that coincides with the Euclidean rotation about the
line. \figref{CircleRotation} (\figloc{right}) illustrates the
\(180^\circ\) rotation about the oriented circle defined as the
rotation around the oriented line that the circle transforms into by
sending one of the points on the circle to infinity.

\begin{figure}[h]
  \includegraphics[width=\columnwidth]{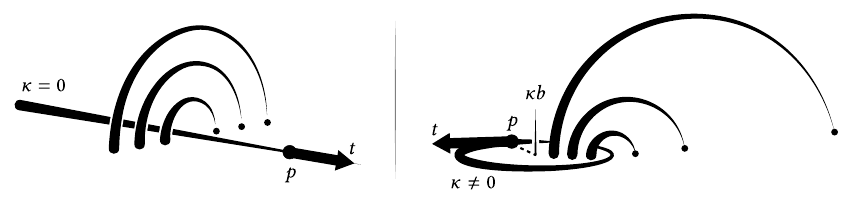}
  \caption{It is useful to describe the \(180^\circ\) rotation around a circle (or line) in \(S^3\) in terms of a point \(p\) on the circle, its tangent at \(p\), and its curvature binormal.}
  \label{fig:CircleRotation} 
\end{figure}

In the case when the circle in \(S^3\) is a line in \(\RR^3\) it is
easy to see that the \(180^\circ\) rotation around the line can also
be computed as the composition of the inversion in two planes
orthogonally intersecting in the line and that this doesn't depend on
the choice of planes. M{\"o}bius geometrically, this means that
inversion in a circle can be computed as the composition of the
inversions in any two spheres orthogonally intersecting in the
circle. This can be done explicitly by using the identification of the
space of oriented spheres with \(\Spheres\).

Consider an oriented circle in \(\RR^3\) passing through the origin
with unit tangent \(t\) and mean curvature vector \(\kappa n\). We can
write the \(180^\circ\) rotation about the circle as the orthogonal
intersection of the sphere
\begin{equation}
  \hmgmatrix{S} = \begin{pmatrix}
    n & 0 \\ \kappa & -n
  \end{pmatrix}
\end{equation}
and the plane defined by the binormal \(b = t\times n\)
\begin{equation}
  \hmgmatrix{P} = \begin{pmatrix}
    -b & 0 \\ 0 & -b
  \end{pmatrix}.
\end{equation}
The \(180^\circ\) degree rotation \(C\) about the circle is the
composition of the reflection in the plane \(P\) and the inversion
\(S\) in the sphere:
\begin{equation}
  \hmgmatrix{C} = \hmgmatrix{S}\hmgmatrix{P} = \begin{pmatrix}
    t & 0 \\ -\kappa b & t
  \end{pmatrix}.
\end{equation}
To describe a circle that instead passes through a point \(p\in\RR^3\)
instead of the origin, the \(180^\circ\)-rotation around the
circle will be given by
\begin{equation}
  \hmgmatrix{C} = \hmgmatrix{T}_p\begin{pmatrix}
    t & 0 \\ -\kappa b & t
  \end{pmatrix}\hmgmatrix{T}_p^{-1}.
\end{equation}
A direct computation shows that all such matrices satisfy \(\hmgmatrix{C}^* = -\hmgmatrix{C}\) and \(\hmgmatrix{C}^2 = -\hmgmatrix{I}\). 
In \appref{QSpheres} we also show that any such matrix describes an oriented circle in \(S^3\) and that we can identify the space
\begin{equation}
  \Circles \coloneqq \big\{\hmgmatrix{C}\in\HH^{2\times 2}\,\,\mid\,\, \hmgmatrix{C}^* = -\hmgmatrix{C},\, \hmgmatrix{C}^2 = -\hmgmatrix{I}\big\} \subset \Mob(3).
\end{equation}
with the space of oriented circles in \(S^3\). 
% Furthermore, they all satisfy \(\hmgmatrix{C}^2 = -\hmgmatrix{I}\) and so all of the matrices \(\hmgmatrix{C}\) describing \(180^\circ\)-degree rotations about circles in \(S^3\) are elements of \(\Circles\). 

Notice also that \(\Circles\subset\sp\), allowing us to interpret an element
\(\hmgmatrix{C}\in\Circles\) also as an infinitesimal M{\"o}bius
transformation, generating the one-parameter group of rotations around
the oriented circle.  \figref{CircleVectorField} visualizes the
M{\"o}bius vector field \(\hmgmatrix{V}_{\hmgmatrix{C}}\) describing
the infinitesimal rotation around a circle
\(\hmgmatrix{C}\in\Circles\). The orientation of the circle is encoded
in the sign of \(\hmgmatrix{C}\) since the vector fields
\(\hmgmatrix{V}_{-\hmgmatrix{C}} = -\hmgmatrix{V}_{\hmgmatrix{C}}\)
differ by a sign and so the 1-parameter family of rotations they
generate describe the rotations around the circle in two opposite
directions.

\begin{figure}[h]
  \includegraphics[width=\columnwidth]{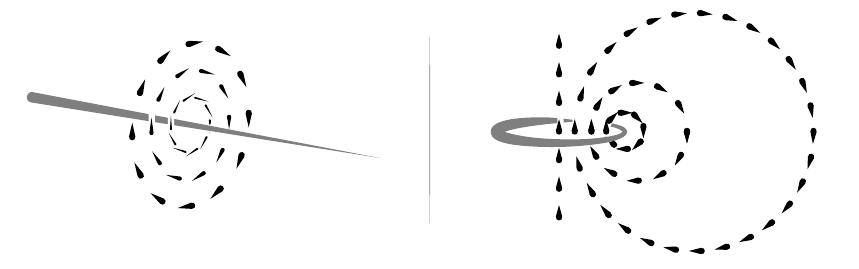}
  \caption{The orientation of a circle associated with \(\hmgmatrix{C}\in\Circles\) can be determined by looking at the M{\"o}bius vector field \(\hmgmatrix{V}_{\hmgmatrix{C}}\) it generates. The vector field describes an infinitesimal rotation around the circle or line and the orientation is chosen so that the rotation is clockwise around the circle.}
  \label{fig:CircleVectorField}
\end{figure}

\paragraph{Space of point pairs}
Lastly, we identify each pair of points with the M{\"o}bius transformation that when restricted to any circle containing the two points is an inflection in the pair of points. We can first map the point pair to the
origin and infinity by a M{\"o}bius transformation. All of the circles
containing the point pair are mapped into lines going through the
origin and the inversion in the pair of points is now just the
Euclidean inversion \(x\mapsto -x\). This does not depend on the
choice of M{\"o}bius transformation used to send the point pair to the
origin and infinity since the Euclidean inversion is invariant under
all stretch rotations (the subgroup of M{\"o}bius transformations
fixing the origin and infinity).

\begin{figure}[h]
  \includegraphics[width=\columnwidth]{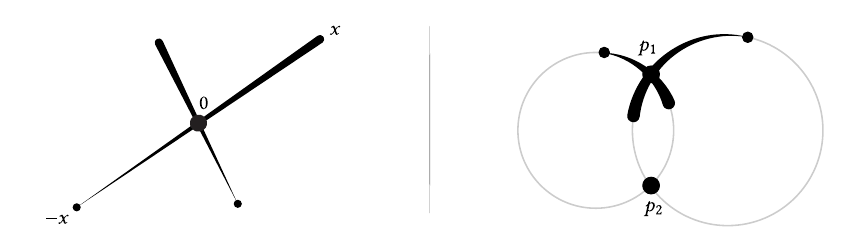}
  \caption{An inversion in a pair of points is the M{\"o}bius
    transformation defined so that the restriction on every circle
    containing the two points is the inversion in a pair of points in
    \(S^1\). When both of the points are in \(\RR^3\) the inversion in them
    can be expressed in terms of one of the points
    \(p_1\) and the difference of the point positions \(p_2 - p_1\).}
\end{figure}

To determine an explicit expression for the
inversion \(\hmgmatrix{U}\in\HH^{2\times 2}\) in a pair of points
\((p_1,p_2)\in \RR^3\times\RR^3\), \(p_1\neq p_2\), let
\(\hmgpoint{p}_i = \hmgcoords{p_i}{1}\) for \(i=1,2\). As a linear
operator acting on homogeneous coordinates, \(\hmgmatrix{U}\) is
defined by its action on a basis of \(\HH^2\). Define the
action of \(\hmgmatrix{U}\) on \(\hmgmatrix{p}_1\) to be equal to minus the
identity and the action of \(\hmgmatrix{U}\) on \(\hmgmatrix{p}_2\) to be
equal to the identity. This clearly defines a M{\"o}bius
transformation that is M{\"o}bius invariantly associated with the pair
of points, and one can readily check that it agrees with the inversion
in the pair of points by sending them to the canonical locations of
the origin and infinity. Starting from the explicit ansatz
\begin{equation}
  \hmgmatrix{U} = \hmgmatrix{T}_{p_1}\begin{pmatrix}1 & 0 \\ 2(p_2 - p_1)^{-1} & -1\end{pmatrix}\hmgmatrix{T}_{p_1}^{-1}
  \label{eq:PointPairInversion}
\end{equation}
it is straightforward to verify that
\(\hmgmatrix{U}|_{\hmgpoint{p}_1} = -\hmgmatrix{I}\) and
\(\hmgmatrix{U}|_{\hmgpoint{p}_2} = \hmgmatrix{I}\) by evaluating
\(\hmgmatrix{U}\) on vectors \(\psi_i\in\hmgpoint{p}_i\) for
\(i=1,2\).  All such matrices \(\hmgmatrix{U}\) associated with
oriented point pairs satisfy \(\hmgmatrix{U}^2 = \hmgmatrix{I}\) and
\(\hmgmatrix{U}^* = -\hmgmatrix{U}\). In \appref{QSpheres} we show
that any such matrix describes an oriented point pair in \(S^3\) and
that we can identify the space
	\begin{equation}
		\PointPairs{} \coloneqq \{\hmgmatrix{U}\in \HH^{2\times 2} \,\,|\,\, \hmgmatrix{U}^* = -\hmgmatrix{U},\,\hmgmatrix{U}^2=\hmgmatrix{I}\}\subset\Mob(3).
	\end{equation}
  with the space of oriented point pairs in \(S^3\).

The property of this M{\"o}bius transformation that enables it to
interact well with the inversions in spheres and circles is that the
Euclidean inversion can be computed by composing the inversion
(reflection) in a plane going through the origin and the \(180^\circ\)
rotation about the line orthogonal to the plane through the
origin. 
Thus, since any pair of points can be sent to zero and
infinity by a M{\"o}bius transformation this implies that the
inflection in two points can be computed as the composition of a
sphere inversion and a circle inversion for any sphere-circle pair
which intersect orthogonally in the pair of points (see \prpref{circlesphereprod}).

The space of point pairs also satisfies \(\PointPairs\subset\sp\) and
so we can interpret each point pair as an infinitesimal M{\"o}bius
transformation, generating the 1-parameter subgroup of scaling
transformations after using a M{\"o}bius transformation to send the
points to the origin and infinity. In \figref{PointPairVectorField} we
visualize the vector field
\(\hmgmatrix{V}_{\hmgmatrix{U}} = -\hmgmatrix{V}_{\hmgmatrix{-U}}\)
for the origin-infinity point pair. The vector field is the radial
vector, which is a harmonic vector field.

\begin{figure}[h]
  \includegraphics[width=\columnwidth]{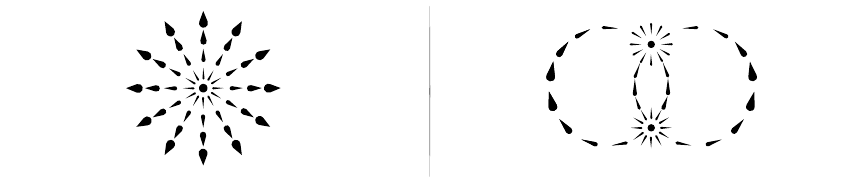}
  \caption{Since \(\PointPairs\subset\sp\) every element \(\hmgmatrix{U}\in\PointPairs\) generates a M{\"o}bius vector field \(\hmgmatrix{V}_{\hmgmatrix{U}}\) that corresponds to an infinitesimal scaling transformation in any affine chart where the pair of points are at zero and infinity.}
  \label{fig:PointPairVectorField}
\end{figure}

Since M{\"o}bius
transformations preserve harmonic vector fields, we deduce that the
vector field associated with an oriented point pair \((p_1,p_2)\) is
also a harmonic vector field on \(S^3\setminus\{p_1,p_2\}\) with a
source and sink at \(p_1\) and \(p_2\),
respectively. Integral curves of this vector field are visualized in \figref{PointPairIntegralCurves}. 

\begin{figure}[ht]
  \centering
  \includegraphics{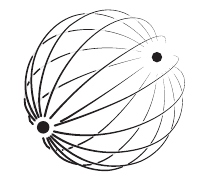}
  \caption{Integral curves of the harmonic vector field associated with an oriented pair of points visualized on a 2-sphere. All of the integral curves are oriented circles going through through the pair of points. The vector field has a source at the point associated with the \(-1\) eigenvalue and a sink at the point associated with the \(+1\) eigenvalue.}
  \label{fig:PointPairIntegralCurves}
\end{figure}

\subsection{Light Cone and Infinitesimal Translations}
\label{sec:LightConeCotangent}
There are two additional spaces that we also need to define
inside of \(\HH^{2\times 2}\): a realization of \(S^3\) itself inside of
\(\HH^{2\times 2}\) and a space of infinitesimal translations. These spaces naturally arise in the classification of pencils of spheres in \secref{SpherePencils}. To realize \(S^3\) inside
of \(\HH^{2\times 2}\) one can take a renormalized limit of the M{\"o}bius
transformations obtained by inversion in spheres of vanishing radius
centered around the points---in this way we can think of each point in
\(S^3\) as a sphere of zero radius.

\paragraph{Light cone model of \(S^3\)} 
The light cone model of M{\"o}bius \(S^3\) can be realized inside of
\(\HH^{2\times 2}\) by considering the five-dimensional vector space
\begin{align}
  \RR^{4,1} & \coloneqq \big\{\hmgmatrix{Y}\in\HH^{2\times 2}\,\,|\,\, \hmgmatrix{Y}^* =\hmgmatrix{Y},\, \tr_{\RR}\hmgmatrix{Y} = 0\big\} \\ 
  & \,= \Big\{
    \begin{pmatrix}
      a & \beta \\ \gamma & -a
    \end{pmatrix}\,\,\mid\,\, a\in\RR^{3},\,\beta,\gamma\in\RR
    \Big\},
\end{align}
endowed with the symmetric bilinear form 
\begin{equation}
  \langle \hmgmatrix{Y},\hmgmatrix{Y}\rangle = -\tfrac12\tr_{\RR}\big(\hmgmatrix{Y}^2\big) = \|a\|^2 - \beta\gamma,\qquad \hmgmatrix{Y} = \begin{pmatrix}a & \beta \\ \gamma & -a\end{pmatrix}
\end{equation}
of signature (4,1).
It is classical that the conformal compactification \(S^3 = \RR^3\cup\{\infty\}\) can be realized as the projectivized light cone in \(\RR^{4,1}\). Let
\begin{equation}
  \LightCone = \{\hmgmatrix{Y}\in\RR^{4,1}\,\,\mid\,\, \sang{\hmgmatrix{Y},\hmgmatrix{Y}} = 0\}
\end{equation}
then 
\begin{equation}
  S^3 \cong \LightCone/\RR^\times
\end{equation}
where the action of \(\RR^\times\) on \(\RR^{4,1}\) is given by scaling.
\begin{remark}
  A direct computation shows that for
  \(\hmgmatrix{A,B}\in\RR^{4,1}\),
  \(\tfrac12\{\hmgmatrix{A},\hmgmatrix{B}\} =
  -\langle\hmgmatrix{A},\hmgmatrix{B}\rangle\hmgmatrix{I}\). Therefore, since the space of oriented 2-spheres \(\Spheres\) is also a subset of
  \(\RR^{4,1}\) it is the Lorentzian unit sphere \(\Spheres = \{\hmgmatrix{S}\in\RR^{4,1}\,\,\mid\,\, \langle\hmgmatrix{S},\hmgmatrix{S}\rangle = 1\}\) inside \(\RR^{4,1}\).
  \label{rem:anticommutator}
\end{remark}

An explicit isomorphism between these two models of M{\"o}bius \(S^3\) is given by the Euclidean lift into the light cone
\begin{equation}
  \Psi_p \coloneqq \hmgmatrix{T}_{p}\begin{pmatrix} 0 & 0 \\ 1 & 0 \end{pmatrix}\hmgmatrix{T}_p^{-1} = \begin{pmatrix}
    p & \|p\|^2 \\ 1 & -p
  \end{pmatrix}.
  \label{eq:EuclideanLift}
\end{equation}
It is straightforward to verify that \(\sang{\Psi_p,\Psi_p} = 0\). 
The point infinity is described by the point 
\begin{equation}
  \infty_{4,1}\coloneqq \begin{pmatrix}0 & 1 \\ 0 & 0\end{pmatrix}
\end{equation}
The map \(\hmgcoords{p}{1}\mapsto [\Psi_p]\), extended so that the image of \(\infty\) is \(\infty_{4,1}\), defines a conformal isomorphism between our quaternionic projective model of \(S^3\) and the classical light cone model of \(S^3\). \eqref{EuclideanLift} is called the Euclidean lift since the identification \(p\in\RR^3\mapsto \Psi_p\in\RR^{4,1}\) is an isometry.
All other lifts into the positive
light cone are obtained by multiplying \(\Psi_p\) by a positive
scalar.

\paragraph{Infinitesimal translations}
To describe geodesic motion inside a parabolic sphere pencil (see \secref{SpherePencils}) one needs to work with a bundle of infinitesimal M{\"o}bius transformations that describe infinitesimal translations when the basepoint is sent to infinity. For each isotropic line \(\hmgpoint{p}\in S^3\) define
\begin{equation}
  \Cotangent_{\hmgpoint{p}} \coloneqq \big\{\hmgmatrix{W}\in\sp\,\,\mid\,\,\ker \hmgmatrix{W}\supset \hmgpoint{p},\,\operatorname{im}\hmgmatrix{W}\subset\hmgpoint{p}\big\}.
\end{equation}
\(\Cotangent_{\hmgpoint{p}}\) and \(\Cotangent_{\hmgmatrix{A}\hmgpoint{p}}\) are related by conjugation by \(\hmgmatrix{A}\) for any \(\hmgmatrix{A}\in\Mob(3)\).
With a little abuse of notation, for a point \(p\in\RR^3\) we will write \(\Cotangent_{p}\) to denote \(\Cotangent_{\hmgcoords{p}{1}}\). It is straightforward to verify that
\begin{equation}
  \Cotangent_{p} = \Big\{\hmgmatrix{T}_{p}\begin{pmatrix}
    0 & 0 \\ w & 0
  \end{pmatrix}
  \hmgmatrix{T}_p^{-1}\,\,\mid\,\, w\in\RR^3\Big\}.
  \label{eq:Vexpression}
\end{equation}
All matrices on the right-hand side satisfy the kernel and image condition defining \(\Cotangent_{p}\), and since both the left and right-hand sides of \eqref{Vexpression} are vector spaces of the same dimension they are equal.
Since \begin{equation}\Cotangent_{\infty} = \Big\{\begin{pmatrix}0 & w \\ 0 & 0 \end{pmatrix}\,\,\mid\,\, w\in\RR^3\Big\}\end{equation} consists of nilpotent matrices 
\begin{equation}
  \exp\begin{pmatrix}0 & w \\ 0 & 0 \end{pmatrix} = \hmgmatrix{I} + \begin{pmatrix}0 & w \\ 0 & 0 \end{pmatrix} = \hmgmatrix{T}_{w}
\end{equation}
and therefore elements of \(\Cotangent_p\) can be identified with infinitesimal translations of \(\RR^3\) after \(p\in\RR^3\) is sent to infinity by a M{\"o}bius transformation. 

\subsection{Configurations of Pairs of 2-Spheres in \(S^3\)}
\label{sec:SpherePencils}
To understand the geometry of rolling spheres in the discrete setting
we characterize configurations of pairs of spheres in
\(S^3\) up to M{\"o}bius transformation and how
they may be mapped into one another through orthogonal
trajectories. We accomplish this exploiting the algebraic structure of
the space of 2-spheres as described in the previous section.

\begin{figure}[h]
  \centering
  \includegraphics[width=.75\columnwidth]{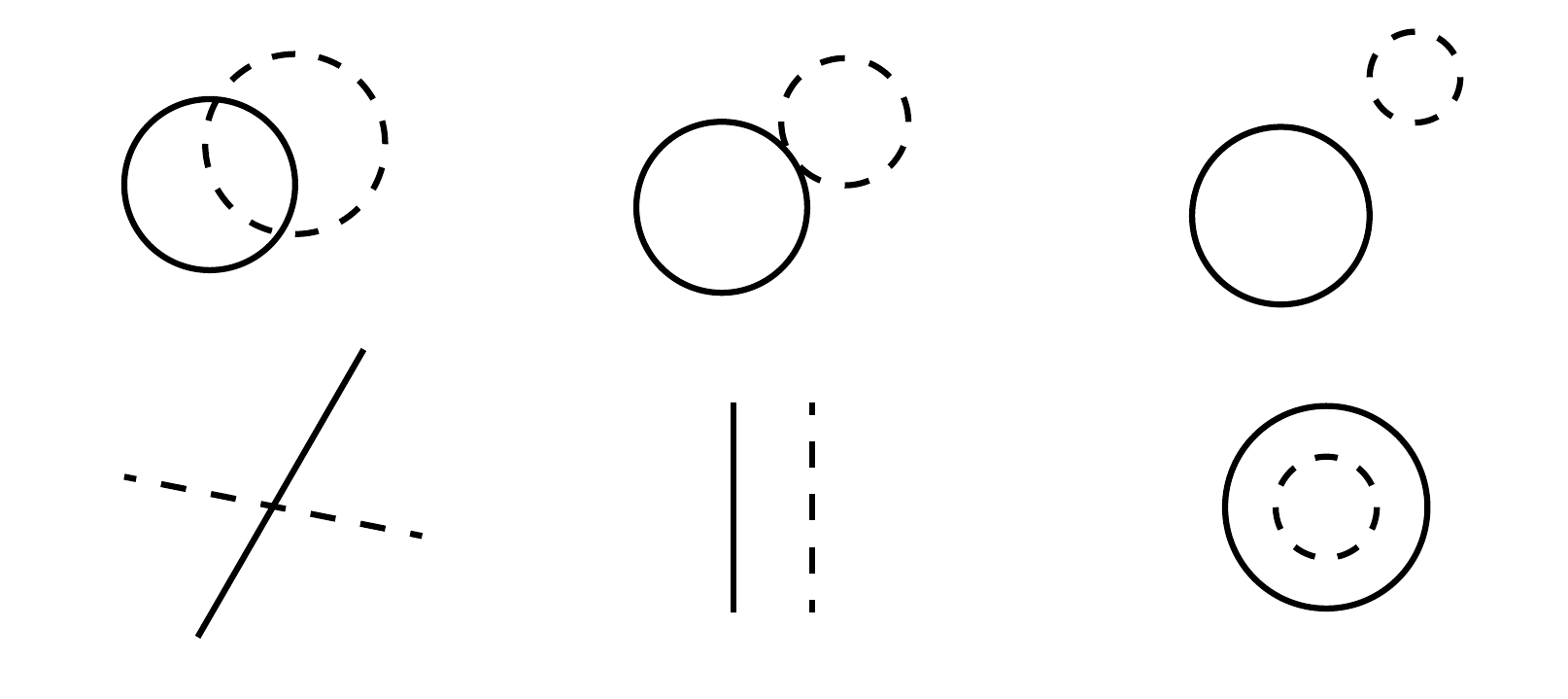}
  \caption{Pairs of circles in \(S^2\) can always be arranged by a M{\"o}bius transformation into a canonical form where the pencil of circles that they generate takes a simple form. \figloc{Left:} if the circles intersect in two points then they generate an elliptic pencil of circles that can be arranged to look like Euclidean lines intersecting in a single point. \figloc{Middle:} if the circles intersect in a single point then they generate a parabolic pencil that can be arranged to look like parallel Euclidean lines. \figloc{Right:} if the circles are disjoint then they generate a hyperbolic pencil that can be arranged to look like Euclidean concentric circles.}
  \label{fig:CircleConfigurations}
\end{figure} 

To understand the possible configurations of pairs of 2-spheres it is
useful to first consider configurations of pairs of circles in
\(S^2\). These can always be arranged by a M{\"o}bius transformation
into one of three configurations illustrated in
\figref{CircleConfigurations}: they can be (1) lines intersecting in a
single point, (2) parallel lines,
or (3) concentric circles. These
are characterized by the number of intersection
points, two, one, and none. More coarsely, configurations of pairs of circles in the plane
are M{\"o}bius equivalent to either concentric circles or a pair of
lines. A completely analogous characterization of configurations of
pairs of 2-spheres in \(S^3\) exists as well: configurations of pairs
of 2-spheres in \(S^3\) can always be arranged by a M{\"o}bius
transformation into one of the following three situations. They can be
\begin{enumerate}
  \item planes that intersect in a line 
  \item parallel planes, 
  \item concentric spheres. 
\end{enumerate}
As in the case of circles in \(S^2\) the configurations of pairs of
2-spheres are characterized by their intersections: in case (i) the two spheres
intersect in a circle and we say that they lie in an elliptic sphere
pencil. An elliptic sphere pencil consists of all 2-spheres containing
a shared circle. in case (ii) we say
that two spheres lie in a parabolic sphere pencil intersecting in only
a single point. Such a pencil consists of all 2-spheres
going through a shared intersection point with the same normal vector
there. Parallel planes form a parabolic sphere pencil with the
intersection point at infinity. In case (iii) the two spheres have an empty
intersection and we say that they lie in a hyperbolic sphere pencil.
Such a pencil consists of all 2-spheres about which a sphere inversion
swaps a shared pair of points. Concentric spheres about the origin
make up the hyperbolic sphere pencil determined by, and swapping, zero
and infinity.

\paragraph{Canonical Form}
If the pair of spheres have non-empty intersection then we can use a
M{\"o}bius transformation to send one of their intersection points to
infinity and in the process transform both of the spheres into
planes. The planes are either parallel or they intersect in a line.

\begin{figure}[h]
  \centering
  \includegraphics[width=\columnwidth]{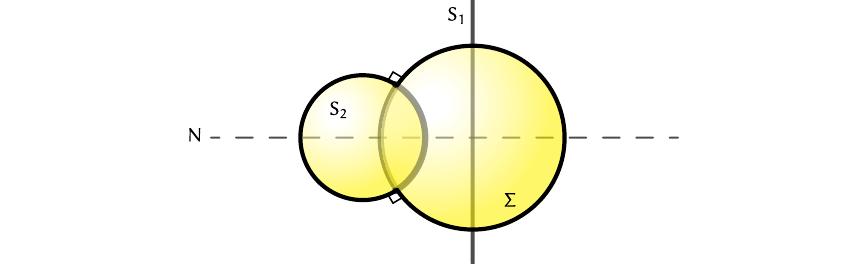}
  \caption{A geometric construction used to show that any pair of disjoint spheres can be arranged by M{\"o}bius transformation of \(S^3\) so that they look like concentric spheres. The dashed line represents the normal circle \(\hmgmatrix{N}\) that intersects \(\hmgmatrix{S}_1\) and \(\hmgmatrix{S}_2\) orthogonally.}
  \label{fig:ConcentricCanonical}
\end{figure}

To see that disjoint 2-spheres can be canonically transformed into
concentric 2-spheres consider the following construction.  Use a
M{\"o}bius transformation to transform one of spheres into a
plane. The second sphere must necessarily still be a round sphere. Let
\(\hmgmatrix{N}\) be the normal line that intersects the plane and the second
sphere orthogonally, and let \(\Sigma\) be a further sphere which
intersects the normal line, plane and second sphere
orthogonally---this can be achieved by making a suitable choice among
the spheres centered on the intersection point of the normal line and
the plane. Now use a M{\"o}bius transformation to send \(\hmgmatrix{N}\) and
\(\Sigma\) to a line and a plane respectively. This transforms the
plane and the second sphere into two spheres that intersect a plane
(transformed \(\Sigma\)) along with a normal line to the plane
orthogonally. As such, the two spheres must be concentric. For a
2-dimensional illustration see \figref{ConcentricCanonical}.

\subsubsection{Parameterizing Pencils of Spheres}
In essence the pencil of spheres connecting a pair of spheres
\(\hmgmatrix{S}_1,\hmgmatrix{S}_2\in\Spheres\) can be parameterized as  
\begin{equation}
  \hmgmatrix{S}(\theta) = \hmgmatrix{S}_1 + \theta \,\hmgmatrix{S}_2
\end{equation}
in each of the three cases.  Even though \(\hmgmatrix{S}(\theta)\) must be
normalized to lie in \(\Spheres\), the unnormalized version still
describes the same M{\"o}bius transformation on \(S^3\). The type of
pencil can be determined by the number of points (spheres of zero
radius) occurring in the pencil: zero for elliptic, one for
parabolic, and two for hyperbolic pencils. Since for spheres
\(\hmgmatrix{S}\in\Spheres\),
\(\langle \hmgmatrix{S},\hmgmatrix{S}\rangle = 1\) and for points \(\Psi\in \LightCone \),
\(\langle \Psi,\Psi\rangle = 0\), the number of
points that appear in the pencil can be decided by looking at the sign
of the discriminant
\begin{equation}
  4\langle \hmgmatrix{S}_1, \hmgmatrix{S}_2\rangle^2 - 4
\end{equation}
of the quadratic equation 
\[
  0 = \langle \hmgmatrix{S}(\theta), \hmgmatrix{S}(\theta)\rangle = 1 + 2\theta\langle \hmgmatrix{S}_1,\hmgmatrix{S}_2\rangle + \theta^2.
\]
For hyperbolic sphere pencils, the parameterization is only valid for a certain range of \(\theta\).

% say something about geodesics
This parameterization of sphere pencils describes geodesics up to parameterization in \(\Spheres\). To directly map 
\(\hmgmatrix{S}_1\) into \(\hmgmatrix{S}_2\) in this manner
one may use the M{\"o}bius transformation
\begin{equation}
  \hmgmatrix{Q} = \hmgmatrix{I} - \hmgmatrix{S}_2\hmgmatrix{S}_1.
\end{equation}
Indeed, since
\[ \hmgmatrix{Q}\hmgmatrix{S}_1 = \hmgmatrix{S}_1 + \hmgmatrix{S}_2 =
  \hmgmatrix{S}_2 \hmgmatrix{Q} \] one finds that
\(\hmgmatrix{Q}\hmgmatrix{S}_1\hmgmatrix{Q}^{-1} = \hmgmatrix{S}_2\).
To see that this transformation describes motion inside the sphere
pencil by orthogonal trajectories we will analyze this transformation
for pairs of spheres in each of the three kinds of sphere pencils
separately. In all cases \(\hmgmatrix{S}_1\times\hmgmatrix{S}_2\) is an element of \(\sp\) describing the M{\"o}bius vector field orthogonal to every sphere of the pencil. 

\subsubsection{Elliptic Sphere Pencils}
\label{sec:ellipticpencil}
Elliptic sphere pencils are the most relevant configuration for the
analysis of the discrete Willmore energy.
% Only configurations of
% pairs of spheres that form an elliptic sphere pencil are relevant to
% the analysis of the discrete Willmore energy.
\begin{figure}[h]
  \centering
  \includegraphics[width=0.4\columnwidth]{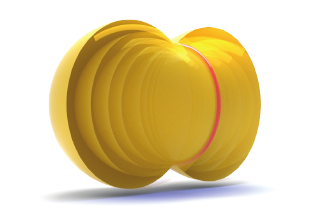}
\end{figure}
Suppose \(\hmgmatrix{S}_1,\hmgmatrix{S}_2\in\Spheres\) generate
an elliptic sphere pencil. Using a M{\"o}bius transformation we may
assume that they describe planes 
intersecting in a line through the origin
\[ \hmgmatrix{S}_i = \begin{pmatrix}n_i & 0 \\ 0 & -n_i\end{pmatrix}\]
for some choice of \(n_i\in S^2\subset\mathbb{R}^3\). Therefore,
\[
  \hmgmatrix{S}_1\hmgmatrix{S}_2 = -\cos\alpha\, \hmgmatrix{I} + \sin\alpha\, \hmgmatrix{C},
\]
where \(\alpha\in[0,\pi)\) is the angle between the two planes
intersecting in the line \(\hmgmatrix{C}\). Here
\(\hmgmatrix{C}\in\Circles\) describes the intersection line
\[ \hmgmatrix{C} = \begin{pmatrix} t & 0 \\ 0 & t\end{pmatrix} \] with
\(t \in S^2 \in\mathbb{R}^3\) the unit vector in the direction
\(n_1\times n_2\).  Since \(\hmgmatrix{C}^2 = -\hmgmatrix{I}\) we can
compute the exponential map in closed form
\[ \exp(\theta\,\hmgmatrix{C}) = \cos\theta~\hmgmatrix{I} +
  \sin\theta~\hmgmatrix{C} \] for \(\theta\in\mathbb{R}\). Since these
are all diagonal matrices the action of \(\exp(\theta \hmgmatrix{C})\)
on points in \(S^3\) is a rotation by
\(2\theta\) around \(\hmgmatrix{C}\). Hence the
M{\"o}bius transformation \(\exp(\frac{\alpha}{2}\hmgmatrix{C})\)
transforms \(\hmgmatrix{S}_1\) onto \(\hmgmatrix{S}_2\) via orthogonal
trajectories. When we say that this describes motion via orthogonal
trajectories we mean that the family of spheres
\[
  \hmgmatrix{S}_t = \exp\left(\tfrac{t\alpha}{2}\hmgmatrix{C}\right)\hmgmatrix{S}_1\exp\left(-\tfrac{t\alpha}{2}\hmgmatrix{C}\right)
\]
from \(t = 0\) to \(t=1\) interpolates between \(\hmgmatrix{S}_1\) and \(\hmgmatrix{S}_2\) inside their shared elliptic sphere pencil. Under this mapping, points in \(\hmgmatrix{S}_1\) follow circular trajectories that are orthogonal to the trajectory of spheres \(\hmgmatrix{S}_t\) in the sphere pencil. So we have shown 
\begin{proposition}
	Let \(\hmgmatrix{S}_1,\hmgmatrix{S}_2\in\Spheres\) be two spheres intersecting in a circle \(\hmgmatrix{C}\). Then
  \begin{equation}
    \sang{\hmgmatrix{S}_1,\hmgmatrix{S}_2} = \cos\alpha,\qquad \hmgmatrix{S}_1\times\hmgmatrix{S}_2 = \sin\alpha\,\hmgmatrix{C}
  \end{equation}
	where \(\alpha\) is the intersection angle between the two spheres. Moreover, the rotation around \(\hmgmatrix{C}\) by an angle \(\alpha\) transforms \(\hmgmatrix{S}_1\) into \(\hmgmatrix{S}_2\) and is equal to the M{\"o}bius transformation \(\exp\big(\frac{\alpha}{2}\hmgmatrix{C}\big)\).
	\label{prp:sphereprod}
\end{proposition}
Applying these insights to the expression \(\hmgmatrix{Q}\) we have
\[
  \hmgmatrix{Q} = \hmgmatrix{I} - \hmgmatrix{S}_2\hmgmatrix{S}_1 = (1 + \cos\alpha)\hmgmatrix{I} + \sin\alpha\, \hmgmatrix{C} = 2\cos\tfrac{\alpha}{2}\big( \cos\tfrac{\alpha}{2}\hmgmatrix{I} + \sin\tfrac{\alpha}{2}\hmgmatrix{C} \big) = 2\cos\tfrac{\alpha}{2}\exp\big(\tfrac{\alpha}{2}\hmgmatrix{C}\big).
\]
Since the scale factor is irrelevant, we see that \(\hmgmatrix{Q}\) is the M{\"o}bius transformation that maps \(\hmgmatrix{S}_1\) into \(\hmgmatrix{S}_2\) via orthogonal trajectories. 

\subsubsection{Parabolic Sphere Pencils}
Suppose \(\hmgmatrix{S}_1,\hmgmatrix{S}_2\in\Spheres\) generate a
parabolic sphere pencil. 
\begin{figure}[h]
  \centering
  \includegraphics[width=0.4\columnwidth]{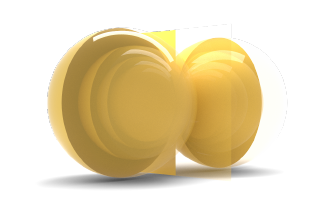}
\end{figure}
Up to a suitable M{\"o}bius transformation we
may assume that they are parallel planes defined by the same normal
vector \(n\in\RR^3\):
\[
  \hmgmatrix{S}_i = \begin{pmatrix}
    n & 2\lambda_i \\ 0 & -n
  \end{pmatrix}
\]
for distinct \(\lambda_i\in\RR\). Here \(\lambda_i\) specifies the planes as the \(\lambda_i\) level sets of \(\langle n,\cdot\rangle\). 
Therefore, 
\[
  \hmgmatrix{S}_1\hmgmatrix{S}_2 = \begin{pmatrix}
    -1 & 2(\lambda_2 - \lambda_1)n \\ 0 & -1
  \end{pmatrix} = -\hmgmatrix{I} + \hat{\hmgmatrix{n}}.
\]
Here \(\hat{\hmgmatrix{n}}\in\mathcal{V}_{\infty}\) is already an infinitesimal translation since the shared intersection point is at infinity. Since
\(\hat{\hmgmatrix{n}}\) is nilpotent
\[
  \exp(\theta \hat{\hmgmatrix{n}}) = \hmgmatrix{I} + \theta\hat{\hmgmatrix{n}} = \begin{pmatrix}
    1 & 2\theta(\lambda_2 - \lambda_1)n \\ 0 & 1
  \end{pmatrix}
\]
is a translation by \(2\theta(\lambda_2 - \lambda_1)\) in the direction from the plane \(\hmgmatrix{S}_1\) to \(\hmgmatrix{S}_2\). In particular, \(\exp\tfrac{\hat{\hmgmatrix{n}}}{2}\) describes the translation mapping \(\hmgmatrix{S}_1\) to \(\hmgmatrix{S}_2\) and so we have shown 
\begin{proposition}
  Let \(\hmgmatrix{S}_1,\hmgmatrix{S}_2\in\Spheres\) be spheres intersecting in a single point \(p\in S^3\) with the same oriented normal there. Then 
  \[
    \langle \hmgmatrix{S}_1, \hmgmatrix{S}_2\rangle = 1 \qquad \hmgmatrix{S}_1\times \hmgmatrix{S}_2 = \hat{\hmgmatrix{n}} \in \mathcal{V}_{p}
    \]
    and \(\hat{\hmgmatrix{n}}\) is an infinitesimal translation when \(p\) is sent to infinity and
    \(\exp\tfrac{\hat{\hmgmatrix{n}}}{2}\) is the M{\"o}bius
    transformation mapping \(\hmgmatrix{S}_1\) to \(\hmgmatrix{S}_2\)
    by orthogonal trajectories.
\end{proposition}
Applying this proposition to the expression \(\hmgmatrix{Q}\) we find that 
\[
  \hmgmatrix{Q} = \hmgmatrix{I} - \hmgmatrix{S}_2\hmgmatrix{S}_1 =
  2\hmgmatrix{I} + \hmgmatrix{S}_1\times \hmgmatrix{S}_2 = 2
  \hmgmatrix{T}_{(\lambda_2 - \lambda_1)n} = 2\exp\tfrac{\hat{n}}{2}
  \]
  which shows that \(\hmgmatrix{Q}\) is the desired M{\"o}bius
  transformation mapping \(\hmgmatrix{S}_1\) to \(\hmgmatrix{S}_2\)
  via orthogonal trajectories.

\subsubsection{Hyperbolic Sphere Pencils}
Now suppose \(\hmgmatrix{S}_1,\hmgmatrix{S}_2\in\Spheres\) describe
disjoint oriented spheres  that generate a hyperbolic sphere
pencil. 
\begin{figure}[h]
  \centering
  \includegraphics[width=0.3\columnwidth]{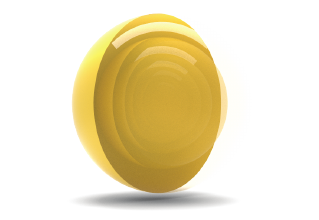}
\end{figure}
By a suitable M{\"o}bius transformation we may assume that
they are concentric spheres centered at the origin
\[\hmgmatrix{S}_i = \begin{pmatrix} 0 & \lambda_i \\
    -\tfrac{1}{\lambda_i} & 0\end{pmatrix}\] for some choice of
\(\lambda_i \in \RR\) for \(i=1,2\). The spheres they represent are
those defined by \(|x|^2 = |\lambda_i|^2\).

Note that the orientations of the two spheres must be compatible, in
the sense that the signs of \(\lambda_1\) and \(\lambda_2\) must be
equal. Otherwise there does not exist an orientation preserving
M{\"o}bius transformation mapping \(\hmgmatrix{S}_1\) to
\(\hmgmatrix{S}_2\). To avoid these undesirable configurations one
needs to assume that the spheres are oriented consistently, and for
simplicity, we will take \(\lambda_1,\lambda_2 > 0\). For disjoint
spheres that are not in canonical form, consistent orientation is
equivalent to the balls bounded by the oriented spheres intersecting in a round ball.

The product of the concentric spheres is
\[
  \hmgmatrix{S}_1\hmgmatrix{S}_2 = \begin{pmatrix} -\lambda_1/\lambda_2 & 0 \\ 0 & - \lambda_2/\lambda_1\end{pmatrix} =  \cosh\sigma\,\hmgmatrix{I} +  \sinh\sigma\,\hmgmatrix{U}
\]
with \(\sigma = \log\frac{\lambda_2}{\lambda_1}\) and where \(\hmgmatrix{U}\in\PointPairs\) is
\[
  \hmgmatrix{U} = \begin{pmatrix}
    1 & 0 \\ 0 & -1
  \end{pmatrix},
  \]
  describing the \((0,\infty)\) point pair. 
\begin{proposition}
  Let \(\hmgmatrix{S}_1,\hmgmatrix{S}_2\in\Spheres\) be disjoint oriented spheres. If the balls bounded by the oriented spheres have non-empty intersection then they generate a hyperbolic sphere pencil defined by a pair of points \(\hmgmatrix{U}\in\PointPairs\) and
  \[
    \langle \hmgmatrix{S}_1, \hmgmatrix{S}_2\rangle = \cosh\sigma \qquad \hmgmatrix{S}_1\times \hmgmatrix{S}_2 = \sinh\sigma~\hmgmatrix{U}
    \]
    for some \(\sigma\in\RR\) and \(\exp(\tfrac\sigma2 \hmgmatrix{U})\) is the M{\"o}bius transformation mapping \(\hmgmatrix{S}_1\) to \(\hmgmatrix{S}_2\) by orthogonal trajectories. 
\end{proposition}
Applying this proposition to the expression \(\hmgmatrix{Q}\) we find that 
\[
  \hmgmatrix{Q} = \hmgmatrix{I} - \hmgmatrix{S}_2\hmgmatrix{S}_1 = (1
  + \cosh\sigma)\hmgmatrix{I} + \sinh\sigma\,\hmgmatrix{U} = 2
  \cosh\tfrac\sigma2 ( \cosh\tfrac\sigma2\, \hmgmatrix{I} +
  \sinh\tfrac\sigma2\, \hmgmatrix{U}) =
    2\cosh\tfrac\sigma2\exp(\tfrac\sigma2 \hmgmatrix{U}).
\]
This shows that \(\hmgmatrix{Q}\) is the desired M{\"o}bius
transformation that maps \(\hmgmatrix{S}_1\) to \(\hmgmatrix{S}_2\)
via orthgonal trajectories.

\subsection{M{\"o}bius Spheres Algebra}
\label{sec:MobiusSpheresAlgebra}
After the consideration of products of spheres in the previous section
we now consider products of circles and point pairs.  

\begin{figure}[h]
  \includegraphics[width=\columnwidth]{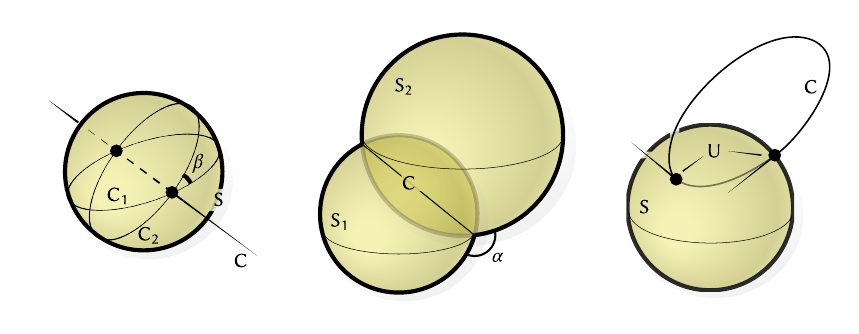}
  \caption{Main operations in the algebra of circles, spheres, and point pairs induced by the realizations into \(\HH^{2\times 2}\).}
\end{figure}
The following result follows from the same reasoning used to prove \prpref{sphereprod}.
\begin{proposition}
  Let \(\hmgmatrix{C}_1,\hmgmatrix{C}_2\in\Circles\) be two circles
  that intersect in two points.
% There is a unique sphere \(S\) containing the two circles.
  Their product is
	\begin{equation}
    \hmgmatrix{C}_1\hmgmatrix{C}_2 = -\sang{\hmgmatrix{C}_1,\hmgmatrix{C}_2}\,\hmgmatrix{I} + \hmgmatrix{C}_1\times\hmgmatrix{C}_2
  \end{equation}
  and 
  \begin{equation}
    \sang{\hmgmatrix{C}_1,\hmgmatrix{C}_2} = \cos\beta,\qquad \hmgmatrix{C}_1\times\hmgmatrix{C}_2 = \sin\beta\,\hmgmatrix{C}
  \end{equation}
	where \(\beta\) is the intersection angle of the two circles
        and \(\hmgmatrix{C}\) is the unique circle normal to \(S\) at
        the two intersection points with \(S\) the unique sphere
        containing both  \(\hmgmatrix{C}_1\) and  \(\hmgmatrix{C}_2\). 
	\label{prp:circleprod}
\end{proposition}

\begin{proposition}
	Consider a sphere \(\hmgmatrix{S}\in\Spheres{}\) along with a pair of points \(\hmgmatrix{U}\in\PointPairs{}\) contained in \(\hmgmatrix{S}\). 
	Then the normal circle \(\hmgmatrix{C}\in\Circles{}\) to the sphere at the two points is equal to the product of the sphere and the point pair
	\begin{equation}
		\hmgmatrix{C} = \hmgmatrix{S} \hmgmatrix{U} = \hmgmatrix{U} \hmgmatrix{S}.
  \end{equation}
	Moreover, \(\hmgmatrix{C}\hmgmatrix{S} = \hmgmatrix{S}\hmgmatrix{C} = -\hmgmatrix{U}\) and \(\hmgmatrix{S} = \hmgmatrix{C}\hmgmatrix{U} = \hmgmatrix{U}\hmgmatrix{C}\).
	\label{prp:circlesphereprod}
\end{proposition}
\begin{proof}
  It is straightforward to verify that \(\hmgmatrix{S}\) and
  \(\hmgmatrix{C}\) commute with \(\hmgmatrix{U}\) since the sphere
  and circle go through the points described by
  \(\hmgmatrix{U}\). From \(\hmgmatrix{S}^* = \hmgmatrix{S}\) and
  \(\hmgmatrix{U}^* = -\hmgmatrix{U}\) we deduce that
  \((\hmgmatrix{S}\hmgmatrix{U})^* = -\hmgmatrix{S}\hmgmatrix{U}\) and
  that \((\hmgmatrix{S}\hmgmatrix{U})^2 = -\hmgmatrix{I}\), and so
  \(\hmgmatrix{S}\hmgmatrix{U}\in\Circles\) and describes a circle
  going through the points \(i\) and \(j\). Since the diagonal entries
  of the matrix representation of a sphere describe its normal
  vectors, whereas the diagonal entries of a circle describe its
  tangent vectors, we conclude that
  \(\hmgmatrix{C} = \hmgmatrix{S}\hmgmatrix{U}\) is the desired normal
  circle orthgonal to \(\hmgmatrix{S}\) through the oriented point
  pair \(\hmgmatrix{U}\).
\end{proof}

\section{Rolling Sphere Connections}
\label{sec:RollingSphereConnections}

With the quaternionic description of spheres in \(S^3\) in place, it
is easy to describe the geometry of a rolling sphere congruence over a
surface. In this section we give the geometric interpretation of
both the smooth and discrete Willmore energies as the curvature of a
connection obtained by rolling spheres over the surface.

Let \(f : M \to \RR^3\) be an immersion of a compact orientable
Riemann surface.  The mean curvature sphere congruence
\(\hmgmatrix{S} : M\to\Spheres\) of the immersion is defined to be
\begin{equation}
  \hmgmatrix{S} = \hmgmatrix{T}_{f}\begin{pmatrix}
    n & 0 \\ -H & -n
  \end{pmatrix}\hmgmatrix{T}_{f}^{-1}
\end{equation}
where \(H\) is the mean curvature of the immersion, and
\(n : M \to S^2\) is the normal of the immersion.  The tangent plane
congruence \(\hmgmatrix{P} : M \to \Spheres\) is given by
\begin{equation} 
  \hmgmatrix{P} = \hmgmatrix{T}_{f}\begin{pmatrix}n & 0 \\ 0 & -n\end{pmatrix}\hmgmatrix{T}_{f}^{-1}.
\end{equation}
Consider the connections
\(\nabla^\hmgmatrix{P} = d -
\tfrac{1}{2}\hmgmatrix{P}\,d{\mkern+1mu}\hmgmatrix{P}\) and
\(\nabla^{\hmgmatrix{S}} = d -
\tfrac{1}{2}\hmgmatrix{S}\,d{\mkern+1mu}\hmgmatrix{S}\) on the trivial
\(\HH^2\)-bundle over \(M\).

The tangent plane congruence is parallel with respect to
\(\nabla^{\hmgmatrix{P}}\)
(\(\nabla^\hmgmatrix{P}\, \hmgmatrix{P} = 0\)) and so the trajectories
of the induced parallel transport move orthogonally to the tangent
planes (see \thmref{OrthogonalTrajectories}). The well-known fact that the curvature of the Levi-Civita
connection yields the Gauss curvature form finds its expression in
the curvature of \(\nabla^{\hmgmatrix{P}}\):
\[ R(\nabla^\hmgmatrix{P}) = -K\hmgmatrix{N}\,\sigma_{f}.\] {Since
  this is an extrinsic description, a rotation about the normal line
  \(\hmgmatrix{N}\) arises as well.}

To extend this picture to the mean curvature spheres we just need to
consider the connection
\(\nabla^{\hmgmatrix{S}} = d
-\tfrac{1}{2}\hmgmatrix{S}\,d{\mkern+1mu}\hmgmatrix{S}\) instead. Just
as before, the mean curvature spheres are parallel with respect to
\(\nabla^{\hmgmatrix{S}}\)
(\(\nabla^{\hmgmatrix{S}}\,\hmgmatrix{S} = 0\)) and so the
trajectories of the parallel transport of \(\nabla^{\hmgmatrix{S}}\)
are orthogonal to the mean curvature spheres. The curvature of
\(\nabla^{\hmgmatrix{S}}\) will now compute the Willmore integrand
\begin{equation}
  R(\nabla^{\hmgmatrix{S}}) = (H^2-K){\hmgmatrix{N}_{\hmgmatrix{S}}}\,\sigma_{f}.
\end{equation} 
{This extrinsic description of the Willmore integrand comes multiplied with \(\hmgmatrix{N}_{\hmgmatrix{S}}\), describing the rotation around a normal circle to the immersion.}

\subsection{Smooth Theory}
\label{sec:SmoothTheory}
To study the connections \(\nabla^{\hmgmatrix{P}}\) and \(\nabla^{\hmgmatrix{S}}\) simultaneously we consider the connection \(\nabla^{\Sigma} = d - \tfrac{1}{2}\Sigma\,d\Sigma\) induced by an arbitrary, but otherwise fixed, tangent sphere congruence \(\Sigma : M \to \Spheres\). Such a sphere congruence is of the form
\begin{equation}
  \Sigma = \hmgmatrix{T}_{f}\begin{pmatrix}
    n & 0 \\ -h & -n
  \end{pmatrix}\hmgmatrix{T}_{f}^{-1}
\end{equation}
for some smooth function \(h : M\to\RR\). 
Since \(-\tfrac{1}{2}\Sigma\,d\Sigma\in\Omega^1(M;\sp)\) the parallel transport of \(\nabla^{\Sigma}\) along a path \(\gamma : [0,1] \to M\) is a M{\"o}bius transformation \(P_{\gamma} \in \Sp\) of \(S^3\). 
\begin{figure}[h]
  \includegraphics[width=\columnwidth]{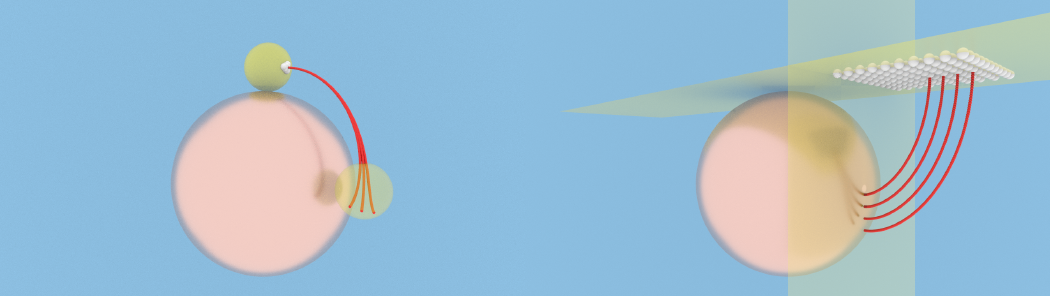}
  \caption{Using the connection \(\nabla^{\Sigma}\) to identify infinitesimally close spheres of \(\Sigma\) by M{\"o}bius transformations we find that the parallel transport of points in \(S^3\) follows trajectories that are orthogonal to the sphere congruence. Visualized in red is the parallel transport of three points obtained by rolling the tangent plane congruence (\figloc{right}) or a tangent sphere congruence (\figloc{left}) over a sphere.}
  \label{fig:OrthogonalTrajectories}
\end{figure}
\begin{proposition}
  The sphere congruence \(\Sigma\) is parallel with respect to \(\nabla^{\Sigma}\).
  \label{prp:SigmaParallel}
\end{proposition}
\begin{proof}
  Recall that the covariant derivative of an endomorphism field \(A\in\Gamma\End\HH^2\) is defined naturally by \((\nabla^{\Sigma}A)\psi \coloneqq \nabla^{\Sigma}(A\psi) - A\nabla^{\Sigma}\psi.\) 
  For a nowhere vanishing section \(\psi:M\to\HH^2\)  
	\begin{align}
		(\nabla^{\Sigma}\Sigma)\psi & = \nabla^{\Sigma}(\Sigma\psi) - \Sigma\nabla^{\Sigma}\psi \\ 
    & = \big(d-\tfrac{1}{2}\Sigma\,d{\mkern+1mu}\Sigma\big)(\Sigma\psi) - \Sigma\big(d-\tfrac{1}{2}\Sigma\,d{\mkern+1mu}\Sigma\big)\psi 
    \\ & = d(\Sigma\psi) - \Sigma\,d\psi - \tfrac{1}{2}d{\mkern+1mu}\Sigma\,\psi - \tfrac{1}{2}d{\mkern+1mu}\Sigma\,\psi 
    \\ & = 0.
	\end{align}
\end{proof}
The geometry of \(\nabla^{\Sigma}\) is described by the following result.
\begin{theorem}
    Let \(\Sigma\) be a sphere congruence over an interval \([0,1]\) and \(\nabla^\Sigma\) be the connection on \([0,1]\times \mathbb H^2\) given by \(\nabla^\Sigma = d - \tfrac{1}{2}\Sigma\,d\Sigma\). Furthermore, let \(\psi\colon [0,1]\to \mathbb H^2\) be parallel such that \(\psi_0\) lies on \(\Sigma_0\). Then \(\psi_t\) lies on \(\Sigma_t\) for all \(t\) and the trajectory intersects the spheres orthogonally.
  \label{thm:OrthogonalTrajectories}
\end{theorem}
\begin{proof}
    Since \(\psi_0\) lies on \(\Sigma_0\) we have that \(\Sigma_0\psi_0 = -\psi_0 n_0\) where \(n_0\) is the normal of \(\Sigma_0\) at \(\psi_0\). \(\Sigma\) is parallel by \prpref{SigmaParallel}. Since \(\psi\) is also parallel 
    \begin{equation}
        \nabla^\Sigma_{\frac{\partial}{\partial t}}(\Sigma\psi + \psi n_0) = 0
    \end{equation}
    and so by the existence and uniqueness of ODEs \(\Sigma_t\psi_t = -\psi_t n_0\) for all \(t\in[0,1]\). 

    To simplify subsequent
    computations we translate everything so that for \(t=0\),
    with \[\psi_t = \begin{pmatrix}{f_t}\\{1}\end{pmatrix}\mu_t\qquad\text{we have}\qquad \psi_0 = \begin{pmatrix}{0}\\{1}\end{pmatrix}\]
    where
    \(f:[0,1]\to\RR^3\) with \(f_0=0\).
    To extract the geometric properties of the trajectory in terms of the curve \(f\) we need to determine how \(\Sigma\) acts on \(\psi'\).
    If we differentiate in time 
    \begin{equation}
      \Sigma'\psi + \Sigma\psi' = (\Sigma\psi)' = \psi'(-n_0).
      \label{eq:o1}
    \end{equation}
    Since \(\psi\) is parallel\begin{equation}
      0 =  \psi' - \tfrac12\Sigma\Sigma'\psi \implies \Sigma'\psi = -2\Sigma\psi',
    \end{equation} 
    and hence by \eqref{o1} \(\Sigma\psi' = \psi'n_0\) . The first row of \(\Sigma_0\psi_0' = \psi_0'n_0\) reads \(n_0f_0' = f_0' n_0\). Therefore, \(f_0'\) is a scalar multiple of \(n_0\). Since the point \(t=0\) was arbitrary this shows that the trajectory \(\psi_t\) traced out by parallel transport of \(\nabla^{\Sigma}\) is orthogonal to \(\Sigma_t\) .
  \end{proof}

The curvature tensor \(\hmgmatrix{R}(\nabla^\Sigma)\in\Omega^2(M;\sp)\) is straightforward to compute from the connection 1-form  \(\nabla^{\Sigma} - d = -\tfrac12\Sigma\,d\Sigma\). 
\begin{proposition}
  \begin{equation*}
    \hmgmatrix{R}(\nabla^{\Sigma}) = \tfrac12\hmgmatrix{T}_{f}\begin{pmatrix}
      ((H^2-K) - (H-h)^2)n\sigma_{f} & 0 \\ 
      dh\wedge((H-h)df + q) & ((H^2-K) - (H-h)^2)n\sigma_{f}
    \end{pmatrix}\hmgmatrix{T}_{f}^{-1}
\end{equation*}
	\label{prp:RollingSpheresCurvature}
\end{proposition}
\begin{proof}
  Let \(\hmgmatrix{A}\coloneqq -\tfrac{1}{2}\Sigma\,d{\mkern+1mu}\Sigma\) be the connection 1-form of \(\nabla^{\Sigma}\). The curvature of the connection is equal to 
  \begin{equation}
    \hmgmatrix{R}(\nabla^{\Sigma}) = d{\mkern+1mu}\hmgmatrix{A} + \hmgmatrix{A}\wedge\hmgmatrix{A} = -\tfrac14 d{\mkern+1mu}\Sigma\wedge d{\mkern+1mu}\Sigma.
  \end{equation}
	The derivative of \(\Sigma\) is equal to
	\begin{align}
    d{\mkern+1mu}\Sigma & = \hmgmatrix{T}_f\left[\begin{pmatrix}0 & \df \\ 0 & 0\end{pmatrix}\begin{pmatrix}n & 0 \\ -h & -n\end{pmatrix} + \begin{pmatrix} dn & 0 \\ -dh & -dn \end{pmatrix}-\begin{pmatrix}n & 0 \\ -h & -n\end{pmatrix}\begin{pmatrix}0 & \df \\ 0 & 0\end{pmatrix}\right]\hmgmatrix{T}_f^{-1} \\ &= \hmgmatrix{T}_f\left(\begin{array}{cc}(H-h)\,\df+q & 0 \\ -dh & -(H-h)\,\df-q\end{array}\right)\hmgmatrix{T}_f^{-1}\end{align}
	and since
	\begin{align}\tfrac{1}{2}\df\wedge\df &= n\,\sigma_{f} & \tfrac{1}{2}q\wedge q &= -(H^2-K)\,n\,\sigma_{f} & \df\wedge q &= q\wedge \df =0\end{align}
	we have
	\begin{align}
    -\tfrac{1}{4}d{\mkern+1mu}\Sigma\wedge d{\mkern+1mu}\Sigma
    & =
    \tfrac12\hmgmatrix{T}_{f}\begin{pmatrix}
      ((H^2-K) - (H-h)^2)n\sigma_{f} & 0 \\ 
      dh\wedge((H-h)df + q) & ((H^2-K) - (H-h)^2)n\sigma_{f}
    \end{pmatrix}\hmgmatrix{T}_{f}^{-1}
  \end{align}
\end{proof}

\begin{example}
  That the Gauss curvature can be realized as the curvature of a
  connection obtained by rolling tangent planes around follows from
  \prpref{RollingSpheresCurvature} by taking \(h \equiv 0\):
  \begin{equation}
    \hmgmatrix{R}(\nabla^{\hmgmatrix{P}}) = \tfrac12\hmgmatrix{T}_f\left(\begin{array}{cc}-Kn\,\sigma_{f} & 0 \\ 0 & -Kn\,\sigma_{f}\end{array}\right)\hmgmatrix{T}_f^{-1} = -\tfrac12K\hmgmatrix{N}\,\sigma_{f}
  \end{equation}
  where \(\hmgmatrix{N}\) describes the congruence of normal lines of $f$. 
\end{example}

\begin{example}
  Another consequence of \prpref{RollingSpheresCurvature} is that the Willmore energy is equal to the curvature of a connection obtained by rolling mean curvature spheres over the surface.
  \begin{equation}
  \hmgmatrix{R}(\nabla^{\hmgmatrix{S}})=\tfrac12\hmgmatrix{T}_f\left(\begin{array}{cc}(H^2-K)\,n\,\sigma_{f} & 0 \\ dH\wedge q & (H^2-K)\,n\,\sigma_{f} \end{array}\right)\hmgmatrix{T}_f^{-1} = \tfrac{(H^2-K)}{2}\hmgmatrix{N}_{\hmgmatrix{S}}\,\sigma_{f}
  \label{eq:WillmoreCurvature}
\end{equation}
The normal circle \(\hmgmatrix{N}_{\hmgmatrix{S}}\) appears after dividing this 2-form by $\sigma_{f}$ and subsequent normalization which is only possible away from umbilic points.
\end{example}

\subsection{Discrete Theory}

\begin{figure}[ht]
  \centering
  \includegraphics[width=\columnwidth]{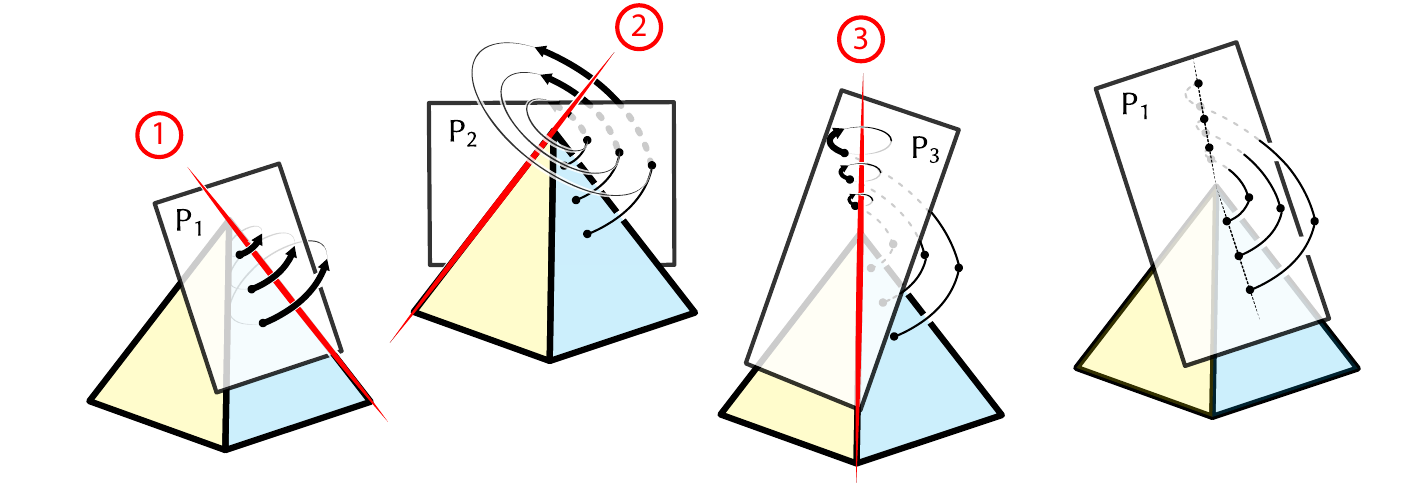}
  \caption{A well-known result is that rolling the tangent planes over a discrete surface computes the angle defect, a discrete version of Gaussian curvature. Up to sending the point of rotation to infinity, this picture also describes the M{\"o}bius rolling of spheres over a discrete surface. In that case, the monodromy angle computes the discrete Willmore energy instead.}
\end{figure}

The smooth interpretation of the Willmore energy provides a geometric principle from which one can arrive at a discretization of the Willmore energy. Given an assignment of discrete mean curvature spheres on the vertices of some cell complex define a discrete Willmore energy per face from the monodromy obtained by rolling the discrete mean curvature spheres over the surface. If the discrete mean curvature spheres are M{\"o}bius invariantly determined from the discrete surface then the resulting discrete Willmore energy is also M{\"o}bius invariant. 

In the following section, we show that the discrete Willmore energy
from \eqref{DiscreteWillmore} arises in this way from the choice of
circumspheres for discrete mean curvature spheres. The circumspheres
are defined per edge as the unique sphere containing the four vertices
of the two faces adjacent to the edge, and so we introduce the Kagome complex (\secref{Kagome}) to describe the combinatorics of rolling
adjacent circumspheres. As we roll the circumspheres around a vertex
we obtain a M{\"o}bius monodromy. Away from the vertices \(i\in\vertices\) where \(\Willmore_i = 0\) the monodromy is a rotation about a normal circle to the circumsphere that we started rolling from. To extract the discrete Willmore energy from this monodromy we only need to look at the rotation angle of this transformation. 

\subsubsection{Simplicial Surfaces}
So that the circumspheres are well-defined we need to assume that the
vertices of the triangles incident on an edge are not concircular.

\begin{definition}
  For each face \(\ijk\in\faces\), the \textbf{circumcircle} \(\Cijk_{\ijk}\in\Circles\) is the unique oriented circle going through \(\fpos_i,\fpos_j,\fpos_k\) in counterclockwise order.   
\end{definition}
The geometric properties of the circumcircle can be summarized in the following expression of \(\Cijk_{\ijk}\) written with respect to the vertex \(\fpos_i\):
\begin{equation}
  \Cijk_{\ijk} = \hmgmatrix{T}_{\fpos_i}\begin{pmatrix}
    \mathtt{t}_{\ijk}^i & 0 \\ -\mathtt{k}_{\ijk} & \mathtt{t}_{\ijk}^i
  \end{pmatrix}\hmgmatrix{T}_{\fpos_i}^{-1}
\end{equation}
where \(\mathtt{t}_{\ijk}^i\) is the tangent to the circumcircle \(\Cijk_{\ijk}\) at the point \(\fpos_i\) and \(\mathtt{k}_{\ijk}\) is the circumcircle curvature binormal.

\begin{definition}
  For each edge \(\ij\in\edges\), the \textbf{circumcircle intersection angle} \(\beta_{\ij}\in[0,\pi)\) is defined to be the intersection angle between the circumcircles \(\Cijk_{\ijk}\) and \(\Cijk_{\jil}\).
\end{definition}
The circumcircle intersection angle can be computed from
\(\cos\beta_{\ij} =
\sang{\Cijk_{\ijk},\Cijk_{\jil}}\). 

\begin{definition}
  For each edge \(\ij\in\edges\), the \textbf{edge circumsphere}
  \(\Sij_{\ij}\) is defined to be the unique sphere containing the
  circumcircles \(\Cijk_{\ijk}\) and \(\Cijk_{\jil}\). The orientation
  of \(\Sij_{\ij}\) is determined so that the outward pointing normal
  to \(\Sij_{\ij}\) at the points \(\fpos_i\) is given by the direction of
  the cross-product of the tangent vectors of the circumcircles
  \(\mathtt{t}_{jil}^i\times \mathtt{t}_{ijk}^i\).
\end{definition}
We will later also need to refer to the geometric properties of the
circumsphere that are summarized in the following expression for
\(\Sij_{\ij}\):
\begin{equation}
  \Sij_{\ij} = \hmgmatrix{T}_{\fpos_i}\begin{pmatrix}
    \mathtt{n}_{\ij}^i & 0 \\ -h_{\ij} & -\mathtt{n}_{\ij}^i
  \end{pmatrix}\hmgmatrix{T}_{\fpos_i}^{-1}
\end{equation}
where \(\mathtt{n}_{\ij}^i\) is the normal to the edge circumsphere
\(\Sij_{\ij}\) at the point \(\fpos_i\) and \(h_{\ij}\) is its mean
curvature.  The circumspheres are a natural choice of the discrete
mean curvature spheres since they are defined in a M{\"o}bius invariant
way using the vertex positions. From the point of view of the discrete Willmore energy, they are also the only natural choice since the energy is defined in terms of circumcircle intersection angles and two adjacent circumcircles uniquely determine the circumsphere.

\subsubsection{The Kagome Complex}
\label{sec:Kagome}
\begin{figure}[h]
	\includegraphics[width=\columnwidth]{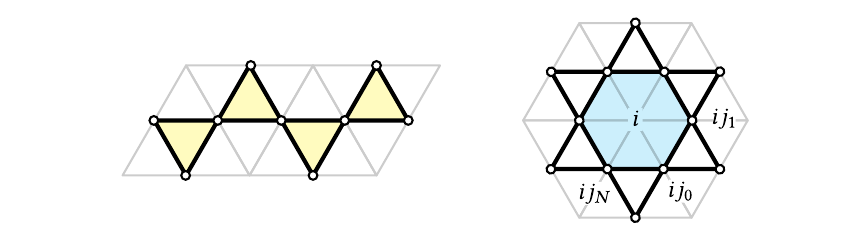}
	\caption{The Kagome complex is obtained by connecting up the edge midpoints of the original simplicial surface. The faces of the Kagome complex inside each triangle can be used to express integrability conditions over the triangles, and the faces of the Kagome complex associated with the vertices can be used to express integrability conditions around vertices.}
	\label{fig:Kagome Complex}
\end{figure}
Since the circumspheres live on the edges of \(\mesh\), the M{\"o}bius
transformations which identify incident circumspheres are most
conveniently described as maps that live on the oriented edges of the
Kagome complex. The vertices of the Kagome complex are identified with
the edges of \(\mesh\) while the edges of the Kagome complex are
identified with the corners of \(\mesh\). A corner of \(\mesh\) is
described by a face \(\ijk\in\faces\) and a vertex \(i\) incident to the
face.  The faces of the Kagome complex are partitioned into two sets,
one is identified with the faces of \(\mesh\)
(highlighted in the left side of \figref{Kagome Complex}) while the
other set is identified with the vertices of \(\mesh\)
(highlighted in the right side of \figref{Kagome Complex}). We write \(j_{ki}\) to denote oriented Kagome edges from \(jk\) to \(ji\). 

The kagome lattice is a two-dimensional geometric lattice structure that has a distinctive pattern of interconnected triangles, reminiscent of a traditional Japanese woven basket called a ``kagome.'' It is named after this basket due to its resemblance to the interwoven, hexagonal, and equilateral faces seen in the basket's weave~\cite{Mekata:2003:KSB}. The Kagome lattice is often used as a theoretical model in condensed matter physics and material science to study phenomena such as magnetism and electron transport~\cite{Syozi:1951:SKL}. We call our complex the Kagome complex in light of the fact that it generalizes the combinatorics of the Kagome lattice to a triangulated surface.
 
\subsubsection{Rolling Circumspheres}
Let us define the rolling circumsphere connection as a discrete \(\Sp\)-connection over the trivial vector bundle over the Kagome complex of \(\mesh\); that is, it is an assignment of a M{\"o}bius transformation of \(S^3\) between fibers \(\HH^2_{\ij} = \HH^2\) and \(\HH^2_{jk} = \HH^2\) of the trivial \(\HH^2\)-bundle over the vertices in the Kagome complex.
\begin{definition}
	The \textbf{rolling circumspheres connection} assigns to each oriented edge \(i_{jk}\) in the Kagome complex the map \(\hmgmatrix{P}(\nabla^{\Sij})_{i_{jk}} \in \Sp\) defined as \[\hmgmatrix{P}(\nabla^{\Sij})_{i_{jk}} \coloneqq\exp\big( \tfrac12\alpha^{i}_{jk} \Cijk_{\ijk} \big) \] where \(\alpha^i_{\jk}\) is the signed angle between the circumspheres \(\Sij_{\ij}\) and \(\Sij_{\ki}\).
\end{definition}
The notation \(\hmgmatrix{P}(\nabla^{\Sij})\) is used to indicate that this is a discrete version of parallel transport induced by the connection \(\nabla^{\hmgmatrix{S}}\). The discrete connection can be used to parallel transport points in \(S^3\) and M{\"o}bius transformations of \(S^3\). Consider a M{\"o}bius transformation \(A\in\HH^{2\times 2}\) of \(S^3\). It can be parallel transported from \(ij\) to \(jk\) by conjugation with \(\hmgmatrix{P}(\nabla^{\hmgmatrix{S}})_{i_{jk}}\). In particular, if we look at the parallel transport of the circumspheres along Kagome edges we find that circumspheres are parallel:
\begin{equation}
  \hmgmatrix{S}_{jk} = \hmgmatrix{P}(\nabla^{\hmgmatrix{S}})_{i_{jk}}\,\hmgmatrix{S}_{\ij}\,\hmgmatrix{P}(\nabla^{\hmgmatrix{S}})_{i_{jk}}^{-1}.
\end{equation}
Since \(\Sij_{\ij}\) and \(\Sij_{\jk}\) intersect in the circumcircle
\(\Cijk_{\ijk}\) they define an elliptic sphere pencil, and so \( \hmgmatrix{I} -
\Sij_{\jk}\Sij_{\ij} =
2\cos\tfrac{\alpha^i_{{\jk}}}{2}~\hmgmatrix{P}(\nabla^{\hmgmatrix{S}})_{i_{jk}}\). Since
the parallel transport is obtained by rotation around the circumcircle
the trajectories traced out by the interpolated parallel transport
maps \(t\mapsto
\exp\big(\tfrac{t}{2}\alpha^{i}_{\jk}\Cijk_{\ijk}\big)\) on each edge
are orthogonal to the circumspheres (see also \secref{ellipticpencil}). 

\paragraph*{Monodromy of the Rolling Circumspheres Connection}
The discrete analog of the curvature of the rolling mean curvature
spheres connection is given by the monodromy of the rolling
circumsphere connection over the faces of the Kagome complex.
The faces of the Kagome complex identified with the faces of
\(\mesh\) have trivial monodromy since all of the parallel transport
maps involved are given by rotations around the same circumcircle. In
particular, the parallel transport maps associated with the oriented
edges of these faces commute.  For the faces of the Kagome complex
identified with vertices of the original mesh, the monodromy
of the rolling circumspheres connection is defined as the product
of the parallel transport in counterclockwise order across the
oriented edges bounding a Kagome face associated with a vertex of \(\mesh\).
The computation of this monodromy angle will follow from the following elementary lemma concerning the geometry of spherical polygons.

\begin{lemma}
  Let \(n_0,\dots,n_{m-1}\in S^2\) be the vertices of a
  spherical polygon. Assume that consecutive vertices are not antipodal. The composition of the
  parallel transport maps on \(S^2\) between successive vertices is
  equal to the clockwise rotation around \(n_0\) by the sum of the
  exterior angles of the polygon.
  \label{lem:sphereprod}
\end{lemma}
\begin{proof}
  For \(\ell = 0,\dots, m-1\) let \(\alpha_{\ell}\in[0,\pi)\) and \(t_\ell\in S^2\) be defined by
  \begin{equation}
    \cos\alpha_\ell = \langle n_\ell, n_{\ell+1}\rangle,\qquad \sin\alpha_\ell~t_\ell = n_\ell\times n_{\ell+1}
  \end{equation}
  where the indices are treated modulo \(m\). The exterior angles \(\beta_{\ell}\in[0,\pi)\) are defined by 
  \begin{equation}
    \cos\beta_\ell = \langle t_{\ell-1},t_{\ell}\rangle.
  \end{equation}
  Define the parallel transport rotations as quaternions \(\rho_{\ell}\coloneqq \exp(\tfrac{\alpha_{\ell}}{2}t_{\ell})\),
  which satisfy \(\rho_{\ell} n_\ell \rho_{\ell}^{-1} = n_{\ell+1}\). Introducing \(\sigma_{\ell}\coloneqq \exp(\tfrac{\beta_{\ell}}{2}n_{\ell})\), one obtains that the ordered product
  \begin{equation}
    \prod_{\ell=0}^{m-1}\rho_{\ell}\sigma_{\ell} = \rho_{m-1}\sigma_{m-1}\cdots\rho_{0}\sigma_{0} = \pm 1
    \label{eq:prod1}
  \end{equation}
  since it is a rotation that fixes both \(n_0\) and \(t_0\). 
  \begin{figure}[ht]
    \centering
    \includegraphics[width=\columnwidth]{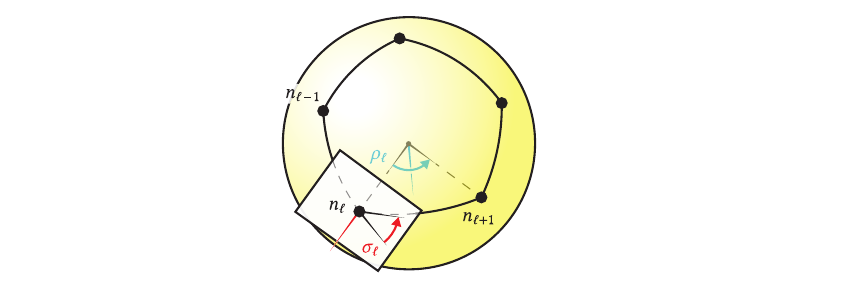}
    \caption{The quaternions \(\rho_{\ell}\) and \(\sigma_{\ell}\) defined in the proof of \lemref{sphereprod} describe the rotations visualized in blue and red, respectively. The axes of rotation are indicated by vectors of the same color.}
  \end{figure}
  Computing
  \begin{equation}
    \rho_\ell\sigma_\ell \rho_{\ell}^{-1} = \rho_\ell(\cos\tfrac{\beta_{\ell}}{2} + \sin\tfrac{\beta_{\ell}}{2}n_\ell)\rho_{\ell}^{-1} = \cos\tfrac{\beta_{\ell}}{2} + \sin\tfrac{\beta_{\ell}}{2}n_{\ell+1} = \exp\Big(\tfrac{\beta_{\ell}}{2}n_{\ell+1}\Big)
  \end{equation}
  and rearranging yields
  \begin{equation}
    \rho_{\ell}\sigma_{\ell} = \exp\Big(\tfrac{\beta_{\ell}}{2}n_{\ell+1}\Big)\rho_{\ell}.
  \end{equation}
  By a cyclic application of this equation, we conclude that
  \begin{align}
    \prod_{\ell=0}^{m-1}\rho_{\ell}\sigma_{\ell} = \exp\Big(\frac{n_0}{2}\sum_{\ell=0}^{m-1}\beta_{\ell} \Big)\prod_{\ell=0}^{m-1}\rho_{\ell},
  \end{align}
  which implies by \eqref{prod1} that
  \begin{equation}
    \prod_{\ell=0}^{m-1}\rho_{\ell} = \pm\exp\Big(-\frac{n_0}{2}\sum_{\ell=0}^{m-1}\beta_{\ell} \Big).
    \label{eq:s2monodromy}
  \end{equation}
  The sign of the quaternion is irrelevant to the rotation it describes and since counterclockwise rotation about \(n_0\) corresponds to a positive angle of rotation the monodromy is a clockwise rotation.
\end{proof}

For the remainder of this section, we will fix an interior vertex \(i\in\vertices\) and let \(\{j_\alpha\}_{\alpha=1}^{N}\) be a labeling of the adjacent vertices in counterclockwise order, with \(N\) the degree of \(i\). 
Define the monodromy of \(\hmgmatrix{P}(\nabla^{\Sij})\) around the vertex \(i\) \begin{equation}\hmgmatrix{M}(\nabla^{\Sij})_i \coloneqq \prod_{\alpha=1}^{N}\hmgmatrix{P}(\nabla^{\Sij})_{i_{j_{\alpha}j_{\alpha+1}}} = \hmgmatrix{P}(\nabla^{\Sij})_{i_{j_{N}j_0}}\cdots\hmgmatrix{P}(\nabla^{\Sij})_{i_{j_1j_{2}}} \hmgmatrix{P}(\nabla^{\Sij})_{i_{j_0j_1}}.
  \label{eq:M1}
\end{equation}
This depends on the labeling of the vertices, but the monodromy
obtained from different choices only differ by conjugation, and so the monodromy angle corresponding to the discrete Willmore energy does not depend on this choice. 

\begin{theorem}
  If \(\Willmore_i\neq 0\) then the monodromy \(\hmgmatrix{M}(\nabla^{\Sij})_i\) is a rotation about a normal circle to \(\Sij_{\ij_0}\) with rotation angle equal to the discrete Willmore energy. If \(\Willmore_i = 0\) then after sending \(\fpos_i\) to infinity by a M{\"o}bius transformation the monodromy is a translation preserving \(\Sij_{\ij_0}\).
\end{theorem}
\begin{proof}
  By sending \(\fpos_i\) to infinity we transform all of the
  circumspheres into planes and the monodromy is transformed into a
  product of Euclidean rotations that maps \(\Sij_{\ij_0}\) to
  itself. Therefore, in this transformed picture
  \(\hmgmatrix{M}(\nabla^{\Sij})_i\) is equal to the composition of a
  translation and a Euclidean rotation about a normal line
  \(\hmgmatrix{N}_{0}\) to the plane \(\Sij_{\ij_0}\) through the
  point \(\fpos_i\). The translation can be written as
  \(\exp\hmgmatrix{V} = \hmgmatrix{I} + \hmgmatrix{V}\) for some
  \(\hmgmatrix{V}\in\Cotangent_{\fpos_i}\) commuting with \(\Sij_{\ij_0}\). Hence,
  \begin{equation}
    \hmgmatrix{M}(\nabla^{\Sij})_i = \exp\Big( \tfrac{\Theta}{2}\hmgmatrix{N}_{0} \Big)\exp\hmgmatrix{V}.
    \label{eq:M2}
  \end{equation}  
  The rotation angle \(\Theta\) can be determined by examining the action of \(\hmgmatrix{M}(\nabla^{\Sij})_i\) on the vector \(\psi_i \coloneqq\begin{psmallmatrix}
    \fpos_i \\ 1
  \end{psmallmatrix}\). 
  % Cotangent space properties and reference here
  On one hand, by \eqref{M2} 
  \begin{equation}
    \hmgmatrix{M}(\nabla^{\Sij})_i\psi_i = \exp\Big( \tfrac{\Theta}{2}\hmgmatrix{N}_{0} \Big)\exp\hmgmatrix{V}\psi_i = \exp\Big( \tfrac{\Theta}{2}\hmgmatrix{N}_{0} \Big)\psi_i = \psi_i\exp\Big( \tfrac{\Theta}{2}n_{0}^i \Big),
  \end{equation}    
  where \(n_{0}^i\) is the normal vector to the circumsphere \(\Sij_0\) at the point \(\fpos_i\). On the other hand, by \eqref{M1} we have that 
  \begin{equation}
    \hmgmatrix{M}(\nabla^{\Sij})_i\psi_i = \prod_{\alpha=1}^{N}\hmgmatrix{P}(\nabla^{\Sij})_{i_{j_{\alpha}j_{\alpha+1}}}\psi_i = \psi_i\prod_{\alpha=1}^{N}\exp\Big( \tfrac{\alpha^i_{j_\ell j_{\ell+1}}}{2} \mathtt{t}_{ij_\ell j_{\ell+1}}^i \Big)
  \end{equation}
  By \lemref{sphereprod} 
  \begin{equation}
    \prod_{\alpha=1}^{N}\exp\Big( \tfrac{\alpha^i_{j_\ell j_{\ell+1}}}{2} \mathtt{t}_{ij_\ell j_{\ell+1}}^i \Big) = \pm\exp\Big(-\tfrac{n_{0}^i}{2}\sum_{\ell=0}^{N}\beta_{ij_\ell}\Big).
  \end{equation}
  Thus, by equating these two expressions
  \begin{equation}
    \exp\Big( \tfrac{\Theta}{2}n_{0}^i \Big) = \pm\exp\Big(-\tfrac{n_{0}^i}{2}\sum_{\ell=0}^{N}\beta_{ij_\ell}\Big).
  \end{equation}
  Therefore,
  \begin{equation}
    \hmgmatrix{M}(\nabla^{\Sij})_i = \mp\exp\Big( -\tfrac{\Willmore_i}{2}\hmgmatrix{N}_{0} \Big)\exp\hmgmatrix{V}.
  \end{equation}
\end{proof}

This geometric interpretation of the discrete Willmore energy as the rotation angle measured when rolling the circumspheres around a vertex mirrors the geometric interpretation of the smooth energy (see \eqref{WillmoreCurvature}). In the discrete case the curvature is an element of \(\Sp\) while in the smooth setting it is an element of \(\sp\). We conclude the paper by discussing one possibility of how the rolling spheres interpretation of the discrete Willmore energy can be used to obtain a discrete Willmore energy for more general piecewise spherical surfaces.

\section{Discussion and Outlook}
\paragraph{M{\"o}bius Invariant Discrete Surfaces}
\begin{figure}[h]
  \includegraphics[width=\columnwidth]{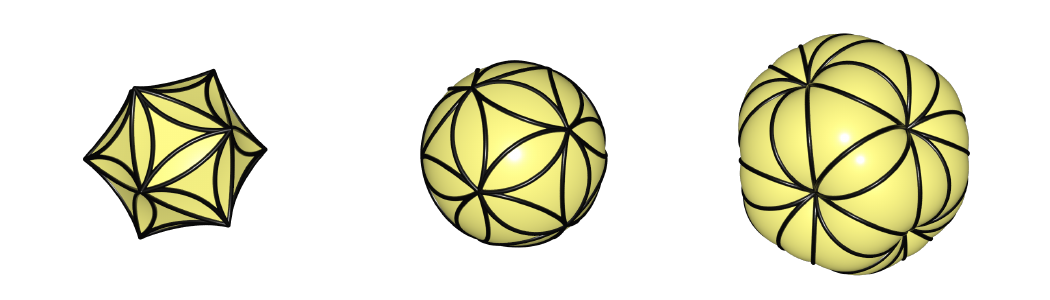}
  \caption{Visualized above is the piecewise spherical surface obtained by filling in faces with the ideal hyperbolic triangle determined by the circumcircle and the choice of sphere congruence per face and the edges are filled in with lenses that are uniquely determined from the piecewise spherical regions in adjacent faces.}
  \label{fig:PiecewiseSphericalSurface}
\end{figure}
Given the additional data of a sphere congruence \(\Sij : \faces \to \Spheres\) such that \(\Sij_{\ijk}\) is in the elliptic sphere pencil generated by the circumcircle \(\Cijk_{\ijk}\) one can produce a M{\"o}bius invariant piecewise spherical surface. That it is M{\"o}bius invariant, means that if we transform the vertex positions and sphere congruence by a M{\"o}bius transformation then the resulting piecewise spherical surface will transform by the same M{\"o}bius transformation. 

The construction we present of a piecewise spherical surface from the additional data of a sphere per face consists of two kinds of piecewise spherical faces: (1) ideal hyperbolic faces associated with faces of \(\mesh\) and (2) lens faces associated with edges of \(\mesh\). The ideal hyperbolic faces are obtained by considering the two regions of \(\Sij_{\ijk}\) bounded by \(\Cijk_{\ijk}\) as Poincar{\'e} models of two-dimensional hyperbolic space. Using the orientation of both the circumcircle and the sphere \(\Sij_{\ijk}\) we can mark the region to the left of the circumcircle as the interior region and we can fill in the vertices \(\fpos_i,\fpos_j,\fpos_k\) with an ideal hyperbolic triangle (which is determined by the three vertices on the ideal boundary) \(\Delta_{\ijk}\) in the interior Poincar{\'e} model of hyperbolic space. For each edge \(\ij\in\edges\) the corresponding boundaries of \(\Delta_{\ijk}\) and \(\Delta_{\jil}\) are circular arc edges that intersect in two points \(\fpos_i,\fpos_j\). As such, there is a unique sphere \(\Sigma_{\ij}\) containing these two circular arc edges and we can take \(\Delta_{\ij}\) to be the spherical lens in \(\Sigma_{\ij}\) interpolating these two circular arc edges. We visualize piecewise spherical surfaces obtained by different choices of sphere congruences in \figref{PiecewiseSphericalSurface}. One natural choice is obtained by taking the harmonic mean of the edge circumspheres as the spheres per face; for each face \(\ijk\):
\begin{equation} 
  \Sij_{\ijk} \coloneqq \frac{\Sij_{\ij} + \Sij_{\jk} + \Sij_{\ki}}{|\Sij_{\ij} + \Sij_{\jk} + \Sij_{\ki}|}\in\Spheres.
  \label{eq:HarmonicMean}
\end{equation}
With this piecewise spherical surface one can define a geometric discretization of the Willmore energy as the monodromy angle of rolling the face spheres \(\Sij_{\ijk}\) onto the edge spheres \(\Sigma_{\ij}\) and continuing all around a vertex of the original mesh---equivalently, one could consider the area of this discrete sphere congruence. Recently, meshes with spherical faces have also been introduced for applications in architectural geometry~\cite{Kilian:2023:MSF}. It is an interesting question to study the properties of this energy and the resulting approximation of the Willmore energy obtained by optimizing away the choice of sphere congruence.

\bibliography{RollingSpheres}

\appendix
    \section{M{\"o}bius Geometry of \(S^3\)}
    \label{app:SpDecomp}
    Looking at the components of the equation \(\hmgmatrix{A}^*\hmgmatrix{A} = \hmgmatrix{I}\) yields the following description of \(\Sp\):
    \begin{equation}
    \Sp = \Bigg\{\begin{pmatrix}
      a & b \\ c & d
    \end{pmatrix}\in\HH^{2\times 2}\,\,\mid\,\, \Re(a\bar{c})= 0, \, \Re(b\bar{d}) = 0,\, \bar{b}c+\bar{d}a = 1\Bigg\}
    \label{eq:SpEquation}
    \end{equation}
    
    \begin{proposition}
      Let \(\hmgmatrix{A}\in\Sp\). Then there exists unique \(x,y\in\RR^3\) and \(\mu\in\HH\) satisfying 
      \begin{equation}
      \hmgmatrix{A} = \begin{pmatrix} 1 & 0 \\ y & 1\end{pmatrix} \begin{pmatrix}\mu & 0 \\ 0 & \bar{\mu}^{-1}\end{pmatrix} \begin{pmatrix} 1 & x \\ 0 & 1 \end{pmatrix}.
      \end{equation}
    \end{proposition}
    \begin{proof}
      Let \(\hmgmatrix{A}\in\Sp\). Then by the characterization from \eqref{SpEquation} there exists \(a,b,c,d\in\HH\) satisfying
      \begin{equation}
        \hmgmatrix{A} = \begin{pmatrix}
          a & b \\ c & d
        \end{pmatrix},\qquad \Re(a\bar{c}) = 0,\,\Re(b\bar{d}) = 0,\, \bar{b}c+\bar{d}a = 1.
      \end{equation}
    
      If \(\hmgmatrix{A}\) fixes \(\infty\) then \(c=0\). Thus, the equation \(\bar{b}c+\bar{d}a = \bar{d}a = 1\) implies that \(d = \bar{a}^{-1}\). Set \(\mu = a\), \(y = 0\), and \(x = \bar{d}b\). Since \(\Re(b\bar{d}) = \Re(\bar{d}b)=\Re(x) = 0\) we have that \(x\in\RR^3\) and \(\mu x = \mu\mu^{-1}b = b \).
      Therefore,
      \begin{equation}
        \begin{pmatrix}
          1 & 0 \\ y & 1
        \end{pmatrix}
        \begin{pmatrix}
          \mu & 0 \\ 0 & \bar{\mu}^{-1}
        \end{pmatrix}
        \begin{pmatrix}
          1 & x \\ 0 & 1
        \end{pmatrix}
        =
        \begin{pmatrix}
          \mu & \mu x \\ 0 & \bar{\mu}^{-1}
        \end{pmatrix}
        =
        \begin{pmatrix}
          a & b \\ 0 & d
        \end{pmatrix}
      \end{equation}
      showing the desired representation for elements of \(\Sp\) fixing \(\infty\).
    
      If \(\hmgmatrix{A}\) does not fix \(\infty\) then define \(p_\infty\in\RR^3\) by \(\hmgmatrix{A}\hmgcoords{1}{0} = \hmgcoords{p_\infty}{1}\) and set 
      \begin{equation}
        \tilde{\hmgmatrix{A}} = 
        \begin{pmatrix}
        1 & 0 \\ -p_\infty^{-1} & 1  
        \end{pmatrix}
        \hmgmatrix{A}.
      \end{equation}
      Since \(\tilde{\hmgmatrix{A}}\hmgcoords{1}{0} = \hmgcoords{1}{0}\) we can apply the results above to find \(\mu\in\HH\) and \(x\in\RR^3\) satisfying 
      \begin{equation}
        \tilde{\hmgmatrix{A}} = 
        \begin{pmatrix}
          \mu & 0 \\ 0 & \bar{\mu}^{-1}
        \end{pmatrix}
        \begin{pmatrix}
          1 & x \\ 0 & 1
        \end{pmatrix}.
      \end{equation}
      Multiplying both sides of the equation by \(\begin{psmallmatrix}
        1 & 0 \\ p_\infty^{-1} & 1  
        \end{psmallmatrix}\) gives the desired representation with \(y = p_\infty^{-1}\)
        \begin{equation}
          \hmgmatrix{A} = 
          \begin{pmatrix}
            1 & 0 \\ p_\infty^{-1} & 1
          \end{pmatrix}
          \begin{pmatrix}
            \mu & 0 \\ 0 & \bar{\mu}^{-1}
          \end{pmatrix}
          \begin{pmatrix}
            1 & x \\ 0 & 1
          \end{pmatrix}.
        \end{equation}
    \end{proof}
    
    \subsection{Quaternionic Realizations of the Space of \(p\)-spheres}
    \label{app:QSpheres}
    In this appendix, we will prove that the equations determined in \secref{MobiusGeometryWQuaternions} that define the spaces \(\Spheres\), \(\Circles\), and \(\PointPairs\) precisely correspond to the inversions in oriented spheres, circles, and point pairs, respectively. The following result will be used to show that the matrices can be assumed to take a simple form where the elementary geometric properties of the spheres can be read off directly from the matrix entries.
    
    \begin{theorem}
      If \(\hmgmatrix{A}\in\Mob(3)\) fixes all points of \(S^3\) then \(\hmgmatrix{A} = \pm\hmgmatrix{I}\).
      \label{thm:S3stabilizer}
    \end{theorem}
    \begin{proof}
      We first show that the claim holds if \(\hmgmatrix{A}\in\Sp\) then we show that there does not exist any orientation reversing M{\"o}bius transformation satisfying the assumption in the theorem. 
    
      Let \(\hmgmatrix{A}\in\Sp\) be a matrix that fixes all points in \(S^3\). In particular, it fixes both zero and infinity and so \(\hmgmatrix{A}\) is a diagonal matrix of the form 
      \begin{equation}
        \hmgmatrix{A} = \begin{pmatrix}
          a & 0 \\ 0 & \bar{a}^{-1}
        \end{pmatrix}
      \end{equation}
      for some \(a\in\HH\). For all \(y\in\RR^3\) we have that 
      \begin{equation}
        \hmgmatrix{A}\begin{pmatrix}
          y \\ 1
        \end{pmatrix} = \begin{pmatrix}
          a y \\ \bar{a}^{-1}
        \end{pmatrix} = \begin{pmatrix}
          y \\ 1
        \end{pmatrix} \lambda_y
      \end{equation}
      for some \(\lambda_y\in\HH\). The second row of this equation implies that \(\lambda_y = \bar{a}^{-1}\) and so the first row of this equation implies that \(ay = y\bar{a}^{-1}\) for all \(y\in\RR^3\). Taking the norm of this equation implies that \(|a| = 1\) and so \(\bar{a}^{-1} = a\). Therefore, \(ay = ya\) for all \(y\in\RR^3\) and so \(a\in\RR\). The condition that \(|a| = 1\) implies that \(a = \pm 1\) and so \(\hmgmatrix{A} = \pm \hmgmatrix{I}\). 
    
      If \(\hmgmatrix{A}\in\Mob(3)\) describes an orientation reversing M{\"o}bius transformation of \(S^3\) then \(\hmgmatrix{A}^*\hmgmatrix{A} = -\hmgmatrix{I}\). If it fixes all points of \(S^3\) then it fixes zero and infinity and so 
      \begin{equation}
        \hmgmatrix{A} = \begin{pmatrix}
          a & 0 \\ 0 & -\bar{a}^{-1}
        \end{pmatrix}
      \end{equation}
      for some \(a\in\HH\). As above, the assumption that it fixes all points in \(S^3\) now implies that \(ay = -y\bar{a}^{-1}\) and taking the norm of this equation implies that \(|a| = 1\) and \(\bar{a}^{-1} = a\). Thus, \(a y = -ya\) for all \(y\in\RR^3\) and this only holds for \(a = 0\), which obviously cannot hold and so we conclude that no orientation reversing M{\"o}bius transformation of \(S^3\) exists that fixes all points of \(S^3\).
    \end{proof}
    
    \begin{proposition}
      Let \(\hmgmatrix{S}\in\Spheres\). Then \(\hmgmatrix{S}\) describes the inversion in a two-sphere in \(S^3\). 
      \label{prp:SpaceOfSpheres}
    \end{proposition}
    \begin{proof}
      Since \(\hmgmatrix{S}^*\hmgmatrix{S} = -\hmgmatrix{I}\), by \thmref{S3stabilizer}, then \(\hmgmatrix{S}\) does not fix all points of \(S^3\). 
      Hence two distinct points are interchanged by \(\hmgmatrix{S}\). Without loss of generality we can assume that zero and infinity are interchanged, and so 
      \begin{equation}
        \hmgmatrix{S} = \begin{pmatrix}
          0 & a \\ b & 0
        \end{pmatrix}
      \end{equation}
      with \(a,b\in\RR\) and \(ab = ba = -1\). So \(a = \rho\) and \(b = -1/\rho\) with \(\rho\in\RR\). This is the inversion in a sphere centered at zero with radius \(\rho\):
      \begin{equation}
        \hmgmatrix{S}\hmgcoord{x}{1} = \hmgcoord{\tfrac{\rho^2}{|x|^2}x}{1}
      \end{equation}
    \end{proof}
    
    \begin{proposition}
      Let \(\hmgmatrix{C}\in\Circles\). Then \(\hmgmatrix{C}\) describes the inversion in a circle in \(S^3\). 
    \end{proposition}
    \begin{proof}
      \(\hmgmatrix{C}\) cannot fix all points of \(S^3\) since by \thmref{S3stabilizer} this would imply that \(\hmgmatrix{C} = \pm \hmgmatrix{I}\) and \(\pm \hmgmatrix{I}\notin\sp\). So two distinct points are interchanged by \(\hmgmatrix{C}\). Without loss of generality we can assume that zero and infinity are interchanged. Then 
      \begin{equation}
        \hmgmatrix{C} = \begin{pmatrix}
          0 & {a} \\ {b} & 0
        \end{pmatrix}
      \end{equation}
      with \({a},{b}\in\RR^3\) and \({ab} = {ba} = -1\). So \({b} = n/\rho\) and \({a} = \rho n\) with \(\rho\in\RR\) and \(n^2 = -1\). This is the inversion in a circle centered at the origin with curvature binormal \(\rho n\):
      \begin{equation}
        \hmgmatrix{C}\hmgcoord{x}{1} = \hmgcoord{-\tfrac{\rho^2}{|x|^2}\bar{n}xn}{1}
      \end{equation}
    \end{proof}
    
    \begin{proposition}
      Let \(\hmgmatrix{U}\in\PointPairs\). Then \(\hmgmatrix{U}\) describes the inversion in a point pair in \(S^3\). 
    \end{proposition}
    \begin{proof}
      \(\hmgmatrix{U}\) cannot fix all points of \(S^3\) since by \thmref{S3stabilizer} this would imply that \(\hmgmatrix{U} = \pm \hmgmatrix{I}\) and \(\pm \hmgmatrix{I}\notin\sp\). So two distinct points that are interchanged by \(\hmgmatrix{U}\). Without loss of generality we can assume that zero and infinity are interchanged. Then
      \begin{equation}
        \hmgmatrix{U} = \begin{pmatrix}
          0 & a \\ b & 0
        \end{pmatrix}
      \end{equation}
      with \(a,b\in\RR^3\) and \(ab = ba = 1\). So \(b = a^{-1}\). This is the inversion in the pair of points \(a\) and \(-a\):
      \begin{equation}
        \hmgmatrix{U}\hmgcoord{a}{1} = \hmgcoord{a}{1},\qquad \hmgmatrix{U}\hmgcoord{-a}{1} = -\hmgcoord{-a}{1}
      \end{equation}
    \end{proof}

\end{document}